\documentclass[a4paper, 9pt, twosides]{amsart}
\usepackage{amsaddr}
\usepackage{graphicx}
\usepackage{enumitem}
\usepackage{dsfont}
\usepackage[usenames, dvipsnames]{color}
\usepackage{todonotes}
\usepackage[colorlinks=true,
		raiselinks=true,
		linkcolor=MidnightBlue,
		citecolor=Mahogany,
		urlcolor=ForestGreen,
		pdfauthor={Atreyee Kundu},
		pdftitle={Stabilizing switching signals},
		pdfkeywords={switched systems, switching signals},
		pdfsubject={Manuscript},
		plainpages=false]{hyperref}

\usepackage{newtxtext}
\usepackage[varg]{newtxmath}

\DeclareFontFamily{T1}{pzc}{}
\DeclareFontShape{T1}{pzc}{m}{it}{<-> [1.2] pzcmi8t}{}
\DeclareMathAlphabet{\mathpzc}{T1}{pzc}{m}{it}

\usepackage{multicol}

\usepackage{mathtools}
\allowdisplaybreaks
\newtheorem{theorem}{Theorem}
\newtheorem{definition}{Definition}

\newtheorem{assump}{Assumption}
\newtheorem{proposition}{Proposition}
\newtheorem{remark}{Remark}

\newtheorem{prob}{Problem}

\newtheorem{lemma}{Lemma}

\newcommand{\norm}[1]{\left\lVert{#1}\right\rVert}
\newcommand{\abs}[1]{\left\lvert{#1}\right\rvert}

\newcommand{\pmat}[1]{\begin{pmatrix}#1\end{pmatrix}}

\renewcommand{\geq}{\geqslant}

\renewcommand{\leq}{\leqslant}

\newcommand{\tendsto}{\rightarrow}

\newcommand{\R}{\mathds{R}}
\newcommand{\N}{\mathds{N}}
\renewcommand{\P}{\mathcal{P}}

\newcommand{\Sw}{\mathcal{S}}
\newcommand{\KL}{\mathcal{KL}}
\newcommand{\Kinfty}{\mathcal{K}_{\infty}}

\newcommand{\Nsw}{\mathrm{N}}
\newcommand{\NswS}{\mathrm{N}^{\mathrm{S}}}
\newcommand{\NswU}{\mathrm{N}^{\mathrm{U}}}

\newcommand{\TswS}{\mathrm{T}^{\mathrm{S}}}
\newcommand{\TswU}{\mathrm{T}^{\mathrm{U}}}

\title[IOSS and state-norm estimation under restricted switching]{On stability and state-norm estimation of switched systems\\under restricted switching}
\author{Atreyee Kundu}
\address{Department of Electrical Engineering,\\Indian Institute of Technology Kharagpur,\\West Bengal - 721302, India,\\ E-mail: atreyee@ee.iitkgp.ac.in}
\thanks{A part of this work was presented at the 60th IEEE Conference on Decision and Control, Texas, Austin (held online), December 2021, see \cite{def}.}


\date{\today}

\begin{document}

	\begin{abstract}
     This paper deals with the analysis of input/output-to-state stability (IOSS) and construction of state-norm estimators for continuous-time switched nonlinear systems under restricted switching. Our contributions are twofold. First, given a family of systems, possibly containing unstable dynamics, a set of admissible switches between the subsystems and admissible minimum and maximum dwell times on the subsystems, we identify a class of switching signals that obeys the given restrictions and preserves IOSS of the resulting switched system.  Second, we design a class of state-norm estimators for switched systems under our class of stabilizing switching signals. These estimators are switched systems themselves with two subsystems --- one stable and one unstable. The key apparatus for our analysis is multiple Lyapunov-like functions. A numerical example is presented to demonstrate the results.
	\end{abstract}
    \maketitle
\section{Introduction}
\label{s:intro}
    The design of state estimators for switched nonlinear systems \cite[Section 1.1.2]{Liberzon} is a numerically difficult task. In the absence of an estimate of full state, an estimate of the norm (i.e., the magnitude) of the system state is often useful in practice, see e.g., \cite{Astolfi-Praly_2005, Krichman_2001,Sontag-Wang_1997} for a detailed discussion. Construction of state-norm estimators for switched systems is closely related to input/output-to-state stability (IOSS) of these systems. IOSS property of a switched system implies that irrespective of the initial state, if the inputs and the observed outputs are small, then the state of the system will become small eventually. In this note we study IOSS and state-norm estimation for continuous-time switched nonlinear systems whose switching signals obey pre-specified restrictions on admissible switches between the subsystems and admissible minimum and maximum dwell times on the subsystems. Such restrictions on switching signals arise in many engineering applications, see e.g., \cite[Remark 1]{abc} for examples.
      
    IOSS of switched differential inclusions under arbitrary switching is studied using common IOSS Lyapunov functions in \cite{Mancilla_2005}.  In \cite{Garcia_2002} the authors prove the existence of state-norm estimators for switched nonlinear systems that are IOSS under arbitrary switching, by employing converse Lyapunov theorem and association of the switched system to a certain nonlinear system whose exogenous inputs take values in a compact set. IOSS and state-norm estimation of a class of hybrid systems that admits a Lyapunov function satisfying an IOSS relation both along the flow and during the jumps, is addressed in \cite{Sanfelice_2010}. In \cite{Liberzon_2012} the authors study IOSS and state-norm estimation of switched systems under average dwell time switching with constrained activation of unstable subsystems \cite[Chapter 3]{Liberzon}. The proposed state-norm estimators are switched systems with one input-to-state stable (ISS) and one unstable subsystem. IOSS of impulsive switched systems both with stabilizing and destabilizing impulses is addressed in \cite{Li_2018}. Recently in \cite{Liu_2022a} 
    ISS and an integral version of it are studied using nonlinear supply functions. Sufficient conditions for integral ISS of switched systems with jumps under slow switching is proposed in \cite{Liu_2022b}.   
%
    
    Our problem setting differs from the existing literature in our choice of the class of switching signals. Indeed, instead of considering all switching signals (arbitrary switching) or a set of stabilizing elements from the set of all switching signals (average dwell time switching), we restrict our attention to a pre-specified subset of the set of all switching signals. We employ multiple Lyapunov-like functions \cite{Branicky_1998} as the key apparatus for our analysis. We show that if (a) it is allowed to switch from every unstable subsystem to a stable subsystem, and (b) a set of scalars computed from Lyapunov-like functions corresponding to the individual subsystems together with an admissible choice of dwell times on the subsystems satisfy a certain inequality, then every switching signal that dwells on the subsystems for the above admissible durations of time and does not activate unstable subsystems in any two consecutive switching instants, preserves IOSS of a continuous-time switched nonlinear system. We also provide a set of sufficient conditions on multiple Lyapunov-like functions and the given ranges of dwell times on the subsystems such that the inequality mentioned in (b) above, is satisfied. In the absence of outputs (resp., also inputs), our class of switching signals ensures ISS (resp., GAS) of a switched system.
    
    Further, we construct a class of state-norm estimators for a switched system operating under our class of stabilizing switching signals. These estimators are switched systems themselves with two subsystems --- one ISS and one unstable. The key apparatus for our construction of state-norm estimators for switched systems is suitable design of the ISS (resp., unstable) dynamics as functions of the rate of decay (resp., growth) of the Lyapunov-like functions corresponding to the IOSS and unstable subsystems, norm of the exogenous inputs and norm of the outputs of the switched system for which estimation is required. Our design of a state-norm estimator is independent of the exact switching instants and destinations of the switched system that the estimator is constructed for. To the best of our knowledge, this is the first instance in the literature where IOSS \emph{and} state-norm estimation of continuous-time switched nonlinear systems is studied under pre-specified restrictions on admissible switches between the subsystems and admissible dwell times on the subsystems.
    
    
    {\bf Notation}. 
   \(\R\) is the set of real numbers.
   \(\norm{\cdot}\) denotes the Euclidean norm.
   For any interval \(I\subseteq[0,\infty[\), \(\norm{\cdot}_{I}\) is the essential supremum norm of a map from \(I\) into some Euclidean space.
   For a statement \(A\), the indicator function \(\mathds{1}_{\{A\}}\) takes the value \(1\) (resp., \(0\)) if the statement \(A\) is true (resp., false). 
   For a real number \(a\), we denote by \(\lfloor a\rfloor\) the largest integer smaller than or equal to \(a\). 
\section{Problem statement}
\label{s:prob_stat}
    We consider the \emph{switched system} \cite[\S1.1.2]{Liberzon}
	   \begin{align}
    \label{e:swsys}
    \begin{aligned}
        \dot{x}(t) &= f_{\sigma(t)}(x(t),v(t)),\\
        y(t) &= h_{\sigma(t)}(x(t)),
    \end{aligned}
    \:\:x(0) = x_{0},\:\:t\geq 0,
    \end{align}
    generated by (a) a family of continuous-time nonlinear systems
    \begin{align}
    \label{e:family}
    \begin{aligned}
        \dot{x}(t) &= f_{p}(x(t),v(t)),\\
        y(t) &= h_{p}(x(t)),
    \end{aligned}
    \:\:x(0) = x_{0},\:\:p\in\P,\:\:t\geq 0,
    \end{align}
    where \(x(t)\in\R^{d}\), \(v(t)\in\R^{m}\) and \(y(t)\in\R^{p}\) are the vectors of states, inputs and outputs at time \(t\), respectively, and \(\P = \{1,2,\ldots,N\}\) is an index set, and
    (b) a \emph{switching signal} \(\sigma:[0,+\infty[\to\P\), which is a piecewise constant function that selects at each time \(t\), the index of the {active subsystem}, i.e., the system from the family \eqref{e:family} that is currently being followed.
     By convention, \(\sigma\) is assumed to be continuous from right and having limits from the left everywhere. Let \(\Sw\) denote the set of all switching signals. We assume that for each \(p\in\P\), \(f_{p}\) is locally Lipschitz, \(f_{p}(0,0) = 0\) and \(h_{p}\) is continuous, \(h_{p}(0) = 0\); the exogenous inputs are Lebesgue measurable and essentially bounded. Thus, a solution to the switched system \eqref{e:swsys} exists in Carath\'eodory sense for some non-trivial time interval containing \(0\) \cite[Chapter 2]{Filippov}.
     
      Let \(\P_{S}\) and \(\P_{U}\) denote the sets of indices of IOSS and unstable subsystems, respectively, \(\P = \P_{S}\sqcup\P_{U}\), and \(E(\P)\) denote the set of all pairs \((p,q)\) such that it is allowed to switch from subsystem \(p\) to subsystem \(q\), \(p,q\in\P\), \(p\neq q\). We let \(0=:\tau_{0}<\tau_{1}<\cdots\) be the \emph{switching instants}; these are the points in time where \(\sigma\) jumps. Let \(\delta\) and \(\Delta\) denote the admissible minimum and maximum dwell time on all subsystems \(p\in\P\), respectively.
    \begin{definition}
    \label{d:adm-sw}
    \rm{
        A switching signal \(\sigma\in\Sw\) is called \emph{admissible} if it obeys the following conditions:
        \(\bigl(\sigma(\tau_i),\sigma(\tau_{i+1})\bigr)\in E(\P)\) and 
            \(\tau_{i+1}-\tau_{i} \in [\delta,\Delta]\),
        \(i=0,1,2,\ldots\).
    }
    \end{definition}

    Let \(\Sw_{\mathcal{R}}\subseteq\Sw\) denote the set of all admissible switching signals \(\sigma\). 
    We will first study IOSS of the switched system \eqref{e:swsys} under the elements of \(\Sw_{\mathcal{R}}\). 
     \begin{definition}{\cite[Appendix A.6]{Liberzon}}
    	\label{d:ioss-cont}
	\rm{
        		The switched system \eqref{e:swsys} is \emph{input/output-to-state stable (IOSS)} under a switching signal \(\sigma\in\mathcal{S}\) if there exist class \(\mathcal{K}_{\infty}\) functions \(\alpha\), \(\chi_{1}\), \(\chi_{2}\) and a class \(\mathcal{KL}\) function \(\beta\) such that for all inputs \(v\) and initial states \(x_{0}\), we have for all \(t\geq 0\) 
        		\begin{align}
        		\label{e:ioss-cont}
            		\alpha(\norm{x(t)}) \leq \beta(\norm{x_{0}},t) + \chi_{1}(\norm{v}_{[0,t]}) + \chi_{2}(\norm{y}_{[0,t]}).
        		\end{align}
        If \(\chi_{2}\equiv 0\), then \eqref{e:ioss-cont} reduces to input-to-state stability (ISS) of \eqref{e:swsys} \cite[Appendix A.6]{Liberzon}, and if also \(v\equiv 0\), then \eqref{e:ioss-cont} reduces to global asymptotic stability (GAS) of \eqref{e:swsys} \cite[Appendix A.1]{Liberzon}.
      }
    \end{definition}
    
    We will first identify the elements \(\sigma\in\Sw_{\mathcal{R}}\subseteq\Sw\) that preserve IOSS of the switched system \eqref{e:swsys}. 
      \begin{prob}
     \label{prob:main1}
        Given a family of systems \eqref{e:family}, the set of admissible switches between the subsystems, \(E(\P)\), and the admissible minimum and maximum dwell times, \(\delta\) and \(\Delta\), on the subsystems \(p\in\P\), respectively, identify a class of admissible switching signals \(\tilde{\Sw}_{\mathcal{R}}\subseteq\Sw_{\mathcal{R}}\) that ensures IOSS of the switched system \eqref{e:swsys}.
     \end{prob}
     
     In the sequel we will call a switching signal \(\sigma\) that preseves IOSS of the switched system \eqref{e:swsys} as a \emph{stabilizing switching signal}. 
     Our second topic of interest is the design of state-norm estimators for the switched system \eqref{e:swsys}. 
     
     Consider a system
     \begin{align}
     \label{e:state_norm}
        \dot{z}(t) = g(z(t),v(t),y(t)),\:\:z(0)=z_0,\:\:t\geq 0
     \end{align}
     whose inputs are the input \(v\) and output \(y\) of the switched system \eqref{e:swsys}, and \(g\) is a locally Lipschitz function. 
     \begin{definition}{\cite[Definition 1]{Liberzon_2012}}
     \label{d:norm_est}
     \rm{
        The system \eqref{e:state_norm} is called a \emph{state-norm estimator} for the switched system \eqref{e:swsys}, if the following conditions hold:
        (a) the system \eqref{e:state_norm} is ISS with respect to \((v,y)\), and
        (b) there exist a class \(\KL\) function, \(\overline{\beta}\) and a class \(\Kinfty\) function, \(\overline{\chi}\), such that for all initial states, \(x_0\) for \eqref{e:swsys} and \(z_0\) for \eqref{e:state_norm}, and all inputs \(v\), we have
                \begin{align}
                \label{e:norm_condn}
                    \norm{x(t)}\leq\overline{\beta}(\norm{x_0}+\norm{z_0},t)+\overline{\chi}(\norm{z(t)})\:\:\text{for all}\:\:t\geq 0.
                \end{align}
        }
     \end{definition}
     We will solve the following problem:
       \begin{prob}
     \label{prob:main2}
        Given a family of systems \eqref{e:family}, the set of admissible switches between the subsystems, \(E(\P)\), and the admissible minimum and maximum dwell times, \(\delta\) and \(\Delta\), on the subsystems \(p\in\P\), respectively, design a state-norm estimator \eqref{e:state_norm} for the switched system \eqref{e:swsys}.
     \end{prob}
     
     We will show that for the switched system \eqref{e:swsys} operating under the elements of \(\tilde{\Sw}_{\mathcal{R}}\), there exist state-norm estimators, which are switched systems themselves, and discuss the design of the same. 
     
\section{Preliminaries}
\label{s:prelims}
     \begin{assump}
    \label{assump:key1}
    \rm{
        There exist class \(\Kinfty\) functions \(\underline{\alpha}\), \(\overline{\alpha}\), \(\gamma_{1}\), \(\gamma_{2}\), continuously differentiable functions \(V_{p}:\R^{d}\to[0,+\infty[\), \(p\in\P\), and constants \(\R\ni\lambda_s,\lambda_u > 0\), such that for all \(\xi\in\R^{d}\) and \(\eta\in\R^{m}\), the following hold:
        \begin{align}
        \label{e:key-prop1}
            \underline{\alpha}(\norm{\xi})&\leq V_{p}(\xi)\leq\overline{\alpha}(\norm{\xi}),\\
        \label{e:key-prop2}
          \frac{\partial V_p}{\partial \xi}f_p(\xi,\eta)&\leq-\lambda_{s}V_{p}(\xi) +\gamma(\eta,h_p(\xi)),\:p\in\P_S,\\
        \label{e:key_prop4} \frac{\partial V_p}{\partial \xi}f_p(\xi,\eta)&\leq\lambda_{u}V_{p}(\xi) + \gamma(\eta,h_p(\xi)),\:p\in\P_U,
        \end{align}
        where \(\gamma(\eta,h_p(\xi)) = \gamma_{1}\bigl(\norm{\eta}\bigr) + \gamma_{2}\bigl(\norm{h_{p}(\xi)}\bigr)\).
        }
    \end{assump}
     \begin{assump}
    \label{assump:key2}
    \rm{
        For each \((p,q)\in E(\P)\) there exist \(\mu {\geq 1}\) such that the functions \(V_p\) and \(V_q\) are related as follows:
        \begin{align}
        \label{e:key-prop3}
            V_{q}(\xi)\leq\mu V_{p}(\xi)\:\:\text{for all}\:\:\xi\in\R^{d}.
        \end{align}
        }
    \end{assump}

     The functions \(V_{p}\), \(p\in\P\), are called the \emph{(multiple) IOSS-Lyapunov-like functions}. Assumptions \ref{assump:key1}-\ref{assump:key2} capture their temporal properties. Condition \eqref{e:key-prop2} is equivalent to the IOSS property for IOSS subsystems \cite{Krichman_2001, Sontag-Wang_1995, Liberzon_2012} and condition \eqref{e:key_prop4} is equivalent to the unboundedness observability property for the unstable subsystems \cite{Angeli-Sontag_1999, Sontag-Wang_1995, Liberzon_2012}. The scalars \(\lambda_{s}\) (resp., \(\lambda_u\)) provide a quantitative measure of stability (resp., instability) of the systems in \eqref{e:family}. Condition \eqref{e:key-prop3} restricts the class of Lyapunov-like functions to be linearly comparable.
    
    We are now in a position to present our results. Proofs of them are presented in a consolidated manner in Section \ref{s:all_proofs}. {\color{black}Owing to space limitations, detailed computations involved in the proofs are presented in \cite{pqr}.}
\section{Main results}
\label{s:mainres}
\subsection{Input/output-to-state stability of switched systems}
\label{ss:ioss}
      We will operate under the following assumptions:
    \begin{assump}
    \label{a:adm-sw}
    \rm{
        For every subsystem \(p\in\P_U\), there exists a subsystem \(q\in\P_S\) such that \((p,q)\in E(\P)\).
    }
    \end{assump}
    \begin{assump}
    \label{a:adm-dw}
    \rm{
        There exist scalars \(\check\delta\), \(\hat\Delta\in[\delta,\Delta]\) such that
        \begin{align}
        \label{e:maincondn}
            -\lambda_s \frac{\check\delta}{\Delta} + \lambda_s \frac{\check\delta}{2\delta} + \lambda_u \frac{\hat\Delta}{2\delta} + (\ln\mu) \frac{1}{\delta} < 0,
        \end{align}
        where \(\lambda_s\), \(\lambda_u\) and \(\mu\) are as described in Assumptions \ref{assump:key1} and \ref{assump:key2}, respectively.
    }
    \end{assump}
    
    Assumption \ref{a:adm-sw} ensures that it is possible to switch from every unstable subsystem to a stable subsystem. Assumption \ref{a:adm-dw} deals with the satisfaction of an inequality involving a set of scalars related to the Lyapunov-like functions of the individual subsystems and certain admissible choice of dwell times on the subsystems.

    \begin{definition}
    \label{d:adm_subset}
    Let \(\Sw'_{\mathcal{R}}\) be the set of all switching signals \(\sigma\in\Sw_{\mathcal{R}}\) that satisfy the following conditions:
    \begin{enumerate}[label = (C\arabic*), leftmargin = *]
        \item\label{prop3} both \(\sigma(\tau_i)\in\P_U\) and \(\sigma(\tau_{i+1})\in\P_U\) are not true, \(i=0,1,2,\ldots\),
        \item\label{prop1} \(\tau_{i+1}-\tau_{i}\in[\check\delta,\Delta]\) whenever \(\sigma(\tau_i)\in\P_S\), \(i=0,1,2,\ldots\), and
         \item\label{prop2} \(\tau_{i+1}-\tau_{i}\in[\delta,\hat\Delta]\) whenever \(\sigma(\tau_i)\in\P_U\), \(i=0,1,2,\ldots\).
    \end{enumerate}
    \end{definition}
    The elements of \(\Sw'_{\mathcal{R}}\subseteq \Sw_{\mathcal{R}}\) do not activate unstable subsystems on any two consecutive switching instants, do dwell on the stable subsystems for at least \(\check\delta\) and at most \(\Delta\) units of time, and dwell on the unstable subsystems for at least \(\delta\) and at most \(\hat\Delta\) units of time. 
    The following result shows that the elements of \(\Sw'_{\mathcal{R}}\) are stabilizing.
    \begin{theorem}
    \label{t:mainres}
    {\it{
        Suppose that Assumptions \ref{assump:key1}-\ref{a:adm-dw} hold. Then the switched system \eqref{e:swsys} is input/output-to-state stable (IOSS) under every switching signal \(\sigma\in\Sw'_{\mathcal{R}}\).
        }}
    \end{theorem}
    
    Our class of stabilizing switching signals, \(\tilde{\Sw}_{\mathcal{R}} = {\Sw}'_{\mathcal{R}}\), is characterized as follows: if (a) the subsystems admit linearly comparable Lyapunov-like functions (Assumptions \ref{assump:key1}-\ref{assump:key2}), (b) the set of admissible switches allows a switch to a stable subsystem from every unstable subsystem (Assumption \ref{a:adm-sw}), and (c) the rate of decay (resp., growth) of Lyapunov-like functions corresponding to stable (resp., unstable) subsystems, the linear comparison factor between these functions and the admissible dwell times on the subsystems together satisfy a certain condition (Assumption \ref{a:adm-dw}), then every admissible switching signal that does not activate unstable subsystems on any two consecutive switching instants (condition \ref{prop3}), and dwells on the subsystems for certain admissible durations of time (conditions \ref{prop1}-\ref{prop2}) ensures IOSS of the continuous-time switched nonlinear system \eqref{e:swsys}.  We will demonstrate in our proof of Theorem \ref{t:mainres} that the choice of the functions \(\alpha\), \(\beta\), \(\chi_1\) and \(\chi_2\) (in Definition \ref{d:ioss-cont}) are independent of \(\sigma\in\Sw'_{\mathcal{R}}\). Consequently, IOSS of the switched system \eqref{e:swsys} is \emph{uniform} over the elements of \(\Sw'_{\mathcal{R}}\) in the above sense. In the absence of outputs (resp., also inputs), the switched system \eqref{e:swsys} is ISS (resp, GAS) under the set of switching signals, \(\Sw'_{\mathcal{R}}\). 
%
    
    \begin{remark}
    \label{rem:restrictive}
    \rm{
        A restriction on activation of unstable dynamics on consecutive switching instants is theoretically restrictive. Indeed, it may be possible to switch carefully between all unstable subsystems resulting in a stable switched system. {\color{black}Global (uniform) asymptotic stability of switched nonlinear systems with all unstable subsystems is studied using discretized Lyapunov function technique in \cite{Xiang_2014} and using contraction analysis in \cite{Yin_2023}. In this note we study input/output-to-state stability of switched nonlinear systems using multiple Lyapunov-like functions and presence of at least one stable subsystem is a requirement for our analysis.}
    }
    \end{remark}
    
    
     We observe the following relations among \(\lambda_s\), \(\lambda_u\), \(\mu\), \(\delta\) and \(\Delta\) that ensure the existence of \(\check\delta\) and \(\hat\Delta\) such that condition \eqref{e:maincondn} is satisfied:
     
     \begin{proposition}
    \label{prop:nextres1}
    {\it{
        If there exist \(\check\delta,\hat\Delta\in[\delta,\Delta]\) such that condition \eqref{e:maincondn} holds, then \(\Delta<2\delta\).
    }
    }
    \end{proposition}
   
         Proposition \ref{prop:nextres1} gives a necessary condition for the satisfaction of \eqref{e:maincondn}. It is, however, easy to see that \(\Delta < 2\delta\) is not a sufficient condition for the satisfaction of \eqref{e:maincondn}. It depends further on the scalars \(\lambda_s\), \(\lambda_u\) and \(\mu\). Indeed, let \(\lambda_s=1.75\), \(\lambda_u=2.17\), \(\mu = 1.25\), \(\delta = 1.5\) and \(\Delta = 2.5 (<2\delta)\). It can be verified that there exists no choice of \(\check\delta,\hat\Delta\in[\delta,\Delta]\) such that condition \eqref{e:maincondn} holds.
    We now present a set of sufficient conditions for the satisfaction of Assumption \ref{a:adm-dw}.
    \begin{proposition}
    \label{prop:mainres2}
    {\it{
        (i) Suppose that \(V_p = V\) for all \(p\in\P\) and
            \(\displaystyle{\frac{\Delta^2}{2\delta^2} < \frac{\lambda_s}{\lambda_s+\lambda_u}}\).
        Then condition \eqref{e:maincondn} holds for all \(\check\delta\) and \(\hat\Delta\in[\delta,\Delta]\).\\
        (ii) Suppose that
            \(\displaystyle{\lambda_u\frac{\Delta}{2\delta}+(\ln\mu)\frac{1}{\delta}<\lambda_s\biggl(\frac{\delta}{\Delta}-\frac{\Delta}{2\delta}\biggr)}\).
        Then condition \eqref{e:maincondn} holds for all \(\check\delta\) and \(\hat\Delta\in[\delta,\Delta]\).\\
        (iii) Suppose that
            \(\displaystyle{\frac{\lambda_u}{2}+(\ln\mu)\frac{1}{\delta} < \lambda_s\biggl(1-\frac{\Delta}{2\delta}\biggr)}\).
        Then \eqref{e:maincondn} holds with \(\check\delta = \Delta\) and \(\hat\Delta=\delta\).\\
        (iv) Suppose that
            \(\displaystyle{\lambda_u\frac{\Delta}{2\delta}+(\ln\mu)\frac{1}{\delta} < \lambda_s\biggl(\frac{\delta}{\Delta}-\frac{1}{2}\biggr)}\).
        Then \eqref{e:maincondn} holds with \(\check\delta = \delta\) and \(\hat\Delta=\Delta\).
    }}
    \end{proposition}

    \begin{remark}
    \label{rem:compa}
    \rm{
        Earlier in \cite{abc} we discussed algorithmic design of \emph{a} stabilizing periodic switching signal that obeys pre-specified restrictions on admissible switches between the subsystems and admissible dwell times on the subsystems. In the current note we identify \emph{a class of} stabilizing switching signals that obeys the given restrictions. 
       Our stability conditions are only sufficient in the sense that if there is no \(\check\delta\), \(\hat\Delta\in[\delta,\Delta]\) such that condition \eqref{e:maincondn} holds and/or the admissible switches between the subsystems violate Assumption \ref{a:adm-sw}, then we cannot conclude non-existence of a non-empty set \(\tilde{\Sw}_{\mathcal{R}}\subseteq\Sw_{\mathcal{R}}\) whose elements are stabilizing.
    }
    \end{remark}
    
\subsection{State-norm estimation of switched systems}
\label{ss:estimation}

     Consider a switched system
    \begin{align}
    \label{e:state_norm2}
        \dot{{z}}(t) = g_{{\zeta}(t)}\bigl({z}(t),v(t),y(t)\bigr),\:\:{z}(0)={z}_0\geq 0,\:\:t\geq 0
    \end{align}
    consisting of the following components:
        a) The subsystems
        \begin{align*}
            \dot{{z}}(t) &= g_0\bigl({z}(t),v(t),y(t)\bigr) = -\lambda_s^*{z}(t)+\overline{\gamma}(t),\\
            \dot{{z}}(t) &= g_1\bigl({z}(t),v(t),y(t)\bigr) = \lambda_u^*{z}(t)+\overline{\gamma}(t),
        \end{align*}
        where \(\overline{\gamma}(t):= \gamma_1\bigl(\norm{v(t)}\bigr)+\gamma_2\bigl(\norm{y(t)}\bigr)\) and \(\lambda_s^*, \lambda_u^*\) satisfy the following conditions:
         \begin{align}
        \label{e:key_condn1a}&\lambda_s^*\in]0,\lambda_s[,\:\:\lambda_u^*\geq\lambda_u,\:\:\lambda_s - \lambda_s^*+\lambda_u - \lambda_u^* \geq 0,\\
        \label{e:key_condn2}&-\lambda_s^*\frac{\check\delta}{\Delta}+\lambda_s^*\frac{\check\delta}{2\delta}+\lambda_u^*\frac{\hat\Delta}{2\delta}
        <0,\\
        \label{e:key_condn3}&(\ln\mu)\frac{1}{\delta}-(\lambda_s-\lambda_s^*)+(\lambda_s - \lambda_s^*+\lambda_u - \lambda_u^*)\frac{\hat\Delta}{2\delta} < 0.
    \end{align}
        b) The switching signal \({\zeta}:[0,+\infty[\to\{0,1\}\) satisfying
        \begin{align*}
            \hspace*{-0.5cm}{\zeta}(t) = 
            \begin{cases}
                0,\:\:\text{if}\:t\in]k(\tilde\delta+\tilde\Delta),k(\tilde\delta+\tilde\Delta)+\tilde\delta],\\
                1,\:\:\text{if}\:t\in]k(\tilde\delta+\tilde\Delta)+\tilde\delta,(k+1)(\tilde\delta+\tilde\Delta)],
            \end{cases}
        \end{align*}
       \(k=0,1,2,\ldots\), where \(\tilde\delta, \tilde\Delta > 0\) satisfy the following conditions:
    \begin{align}
        \label{e:newcondn1}&\tilde\delta \leq \check\delta,\:\:\tilde\Delta\geq\hat\Delta,\\
        \label{e:newcondn4}&-\lambda_s^*+(\lambda_s^*+\lambda_u^*)\frac{\tilde\Delta}{\tilde\delta+\tilde\Delta} < 0,\\
        \label{e:newcondn3}&\frac{\tilde\Delta\hat\Delta}{2\delta}+\frac{\tilde\delta\hat\Delta}{2\delta}-\tilde\Delta
        \leq 0.
    \end{align}
    
    The subsystem \(0\), represented by dynamics \(g_0\), is ISS with respect to the inputs \((v,y)\) and is activated by \(\zeta\) for \(\tilde\delta\) duration of time during the interval \(]k(\tilde\delta+\tilde\Delta),k(\tilde\delta+\tilde\Delta)+\tilde\delta]\) while the subsystem \(1\), represented by dynamics \(g_1\), is unstable and is activated by \(\zeta\) for \(\tilde\Delta\) duration of time during the interval \(]k(\tilde\delta+\tilde\Delta)+\tilde\delta,(k+1)(\tilde\delta+\tilde\Delta)]\), \(k=0,1,\ldots\). Notice that the switching between the subsystems \(g_0\) and \(g_1\) by \(\zeta\) is independent of the exact switching instants and switching destinations of the switching signal \(\sigma\). The following result asserts that the switched system \eqref{e:state_norm2} is a state-norm estimator for the switched system \eqref{e:swsys}. 
    
      \begin{theorem}
    \label{t:mainres3}
    {\it{
       Suppose that Assumptions \ref{assump:key1}-\ref{a:adm-dw} hold. Let the switching signals of \eqref{e:swsys} belong to the set \(\Sw'_{\mathcal{R}}\). Then the switched system \eqref{e:state_norm2} is a state norm estimator for the switched system \eqref{e:swsys}.
        }}
    \end{theorem}
    \begin{remark}
    \label{rem:compa2}
    \rm{
        {\color{black}Our design of state-norm estimator is a restricted switching counterpart of \cite[Theorem 5]{Liberzon_2012}. In both the cases the estimator is a switched system with two subsystems, one stable and one unstable, and the switching signal is independent of the exact knowledge of switching instants and destinations of the switching signals of the switched system \eqref{e:swsys}. The estimator in \cite[Theorem 5]{Liberzon_2012} caters to subsets \(\mathcal{S}'\) of \(\mathcal{S}\) whose elements satisfy a favourably chosen stabilizing (average) dwell time whereas the estimator \eqref{e:state_norm2} caters to switching signals that belong to a pre-specified subset \(\mathcal{S}_{\mathcal{R}}\) of \(\mathcal{S}\).} 
    }
    \end{remark}

        
  
\section{A numerical example}
\label{s:num_ex}
    We consider a family of systems \eqref{e:family} with \(\P = \{1,2,3\}\), where
        \(f_1(x,v) = \pmat{-2x_1+\sin(x_1-x_2)\\-2x_2+\sin(x_2-x_1)+0.5v},\:\:h_1(x) = x_1-x_2\),
        \(f_2(x,v) = \pmat{0.5x_2+0.25\abs{x_1}\\\text{sat}(x_1)+\frac{1}{2}v},\:\:h_2(x) = \abs{x_1}\),
    where \(\text{sat}(x_1)=\min\{1,\max\{-1,x_1\}\}\), 
    \(f_3(x,v) = \pmat{0.2x_1+0.1x_2\\0.3x_1+v}\), \(h_3(x) = x_1\). Clearly, \(\P_S=\{1\}\) and \(\P_U = \{2,3\}\). Let the admissible switches be \(E(\P) = \{(1,2),(1,3),(2,1),(3,1)\}\), the admissible minimum and maximum dwell times, respectively, be \(\delta = 3.5\) and \(\Delta = 4\) units of time.
   Let \(V_1(x)=V_2(x) = \frac{1}{2}(x_1^2+x_2^2)\) and \(V_3(x) = x_1^2+x_2^2\). The following estimates are obtained: \(\lambda_s = 3.5\), \(\lambda_u = 0.73\), \(\mu = 2\), \(\gamma_1(r)=\gamma_2(r) = 2r^2\). Let \(\check\delta = 3.5\), \(\hat\Delta = 4\) units of time. It follows that
    \(\displaystyle{
        -\lambda_s\frac{\check\delta}{\Delta}+\lambda_s\frac{\check\delta}{2\delta}+\lambda_u\frac{\hat\Delta}{2\delta}
        +(\ln\mu)\frac{1}{\delta} = -0.6973}.
    \)
    Thus, Assumptions \ref{assump:key1}-\ref{a:adm-dw} hold. 
    
    We generate \(10\) elements of the set \(\mathcal{S}'_{\mathcal{R}}\) and corresponding to each element, we choose \(x_0\in[-1,1]^2\) and \({\color{blue}v(t)}\in[-0.5,0.5]\) uniformly at random. The resulting processes \(\bigl(\norm{x(t)}\bigr)_{t\geq 0}\) are illustrated till \(t=15\) units of time in Fig. \ref{fig:ioss}. The assertion of Theorem \ref{t:mainres} follows.
\vspace*{-0.5cm}
    \begin{figure}[!htb]
    \centering
    \begin{minipage}{.5\textwidth}
        \centering
        \includegraphics[scale=0.3]{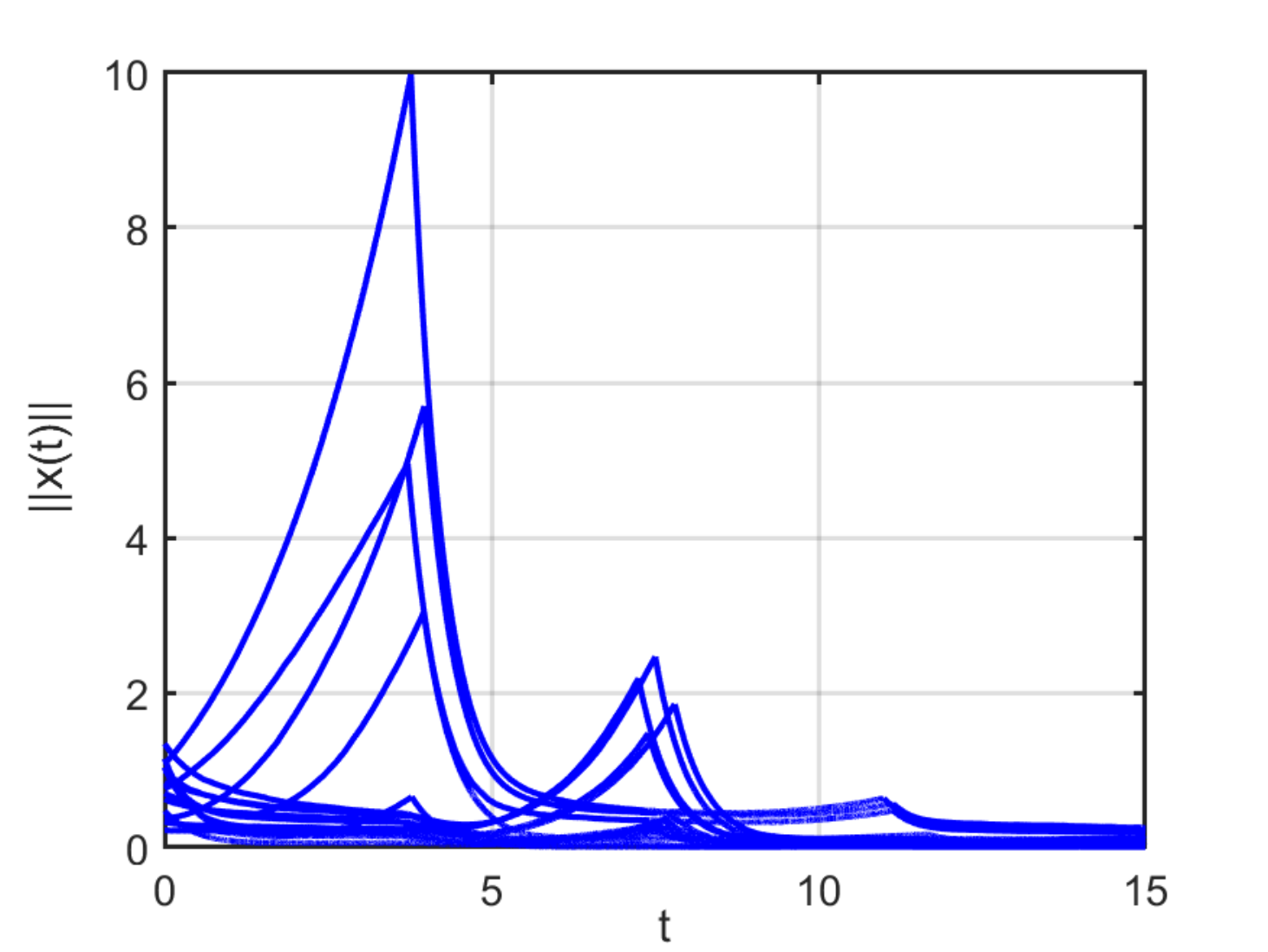}
         \caption{Plot for \(\bigl(\norm{x(t)}\bigr)_{t\geq 0}\).}\label{fig:ioss}
    \end{minipage}%
    \begin{minipage}{0.5\textwidth}
        \centering
        \vspace*{0.5cm}
        \includegraphics[scale=0.3]{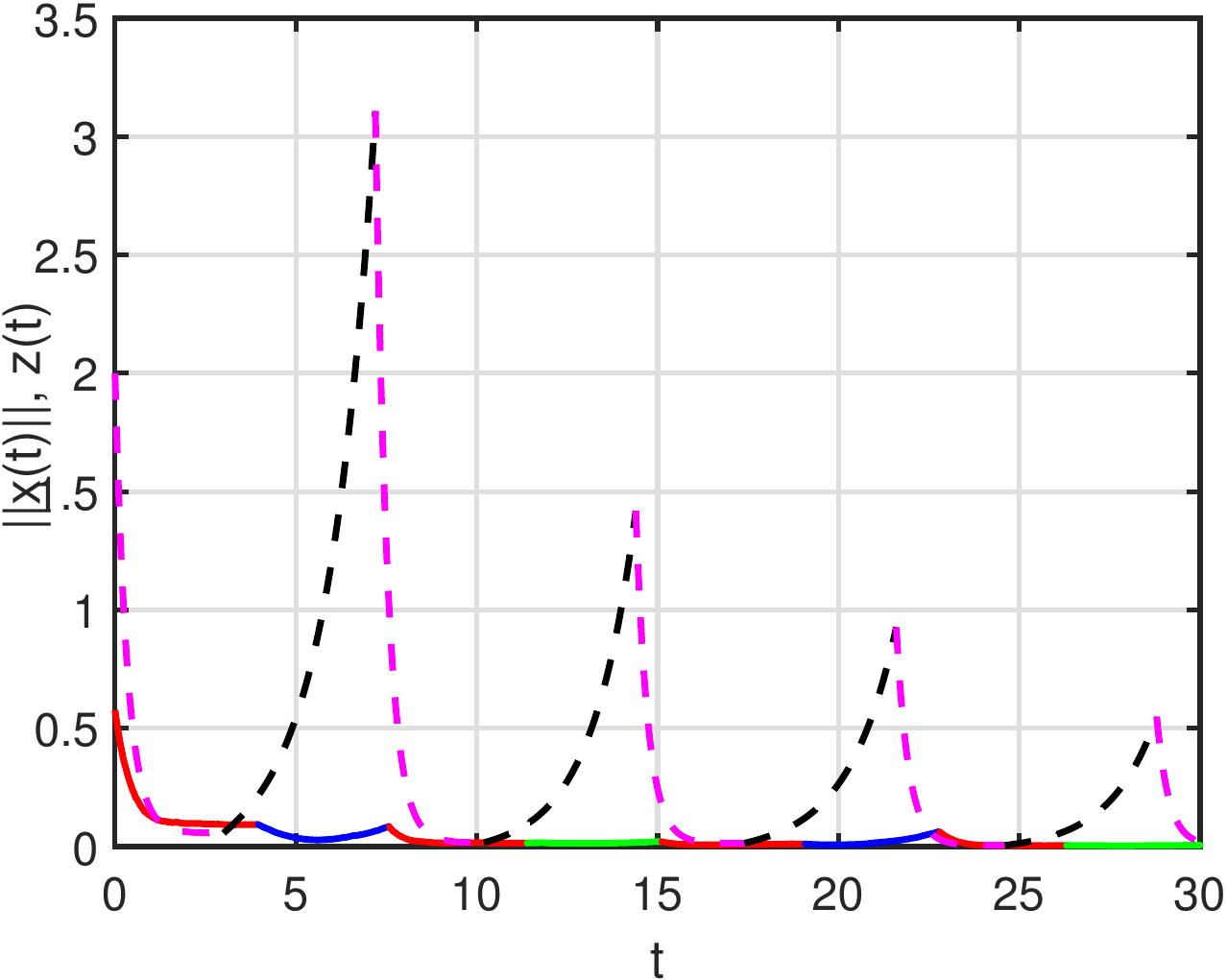}
        \caption{Plot for \(\bigl(\norm{x(t)}\bigr)_{t\geq 0}\) (solid line) and \(\bigl(z(t)\bigr)_{t\geq 0}\) (dashed line).}
        \label{fig:estimate}
    \end{minipage}
    \end{figure}
    \begin{figure}
    \begin{minipage}{.5\textwidth}
        \centering
        \includegraphics[scale=0.3]{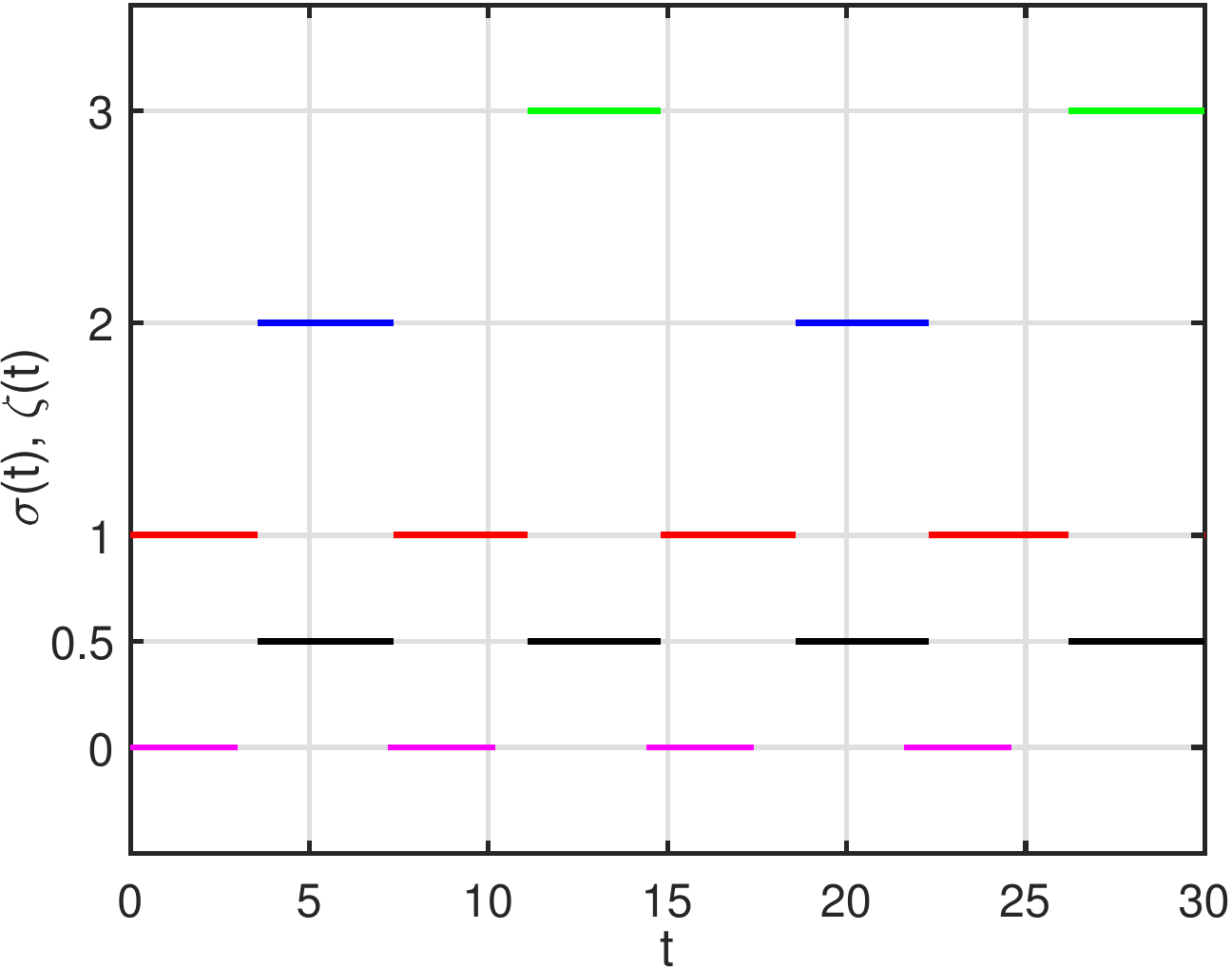}
        \caption{\(\sigma\) and \(\zeta\). \(0\) and \(0.5\)\\denote dynamics \(g_0\) and \(g_1\),\\respectively.}
        \label{fig:estimate_sw}
    \end{minipage}%
    \begin{minipage}{0.5\textwidth}
        \centering
        \vspace*{-0.5cm}
        \includegraphics[scale=0.3]{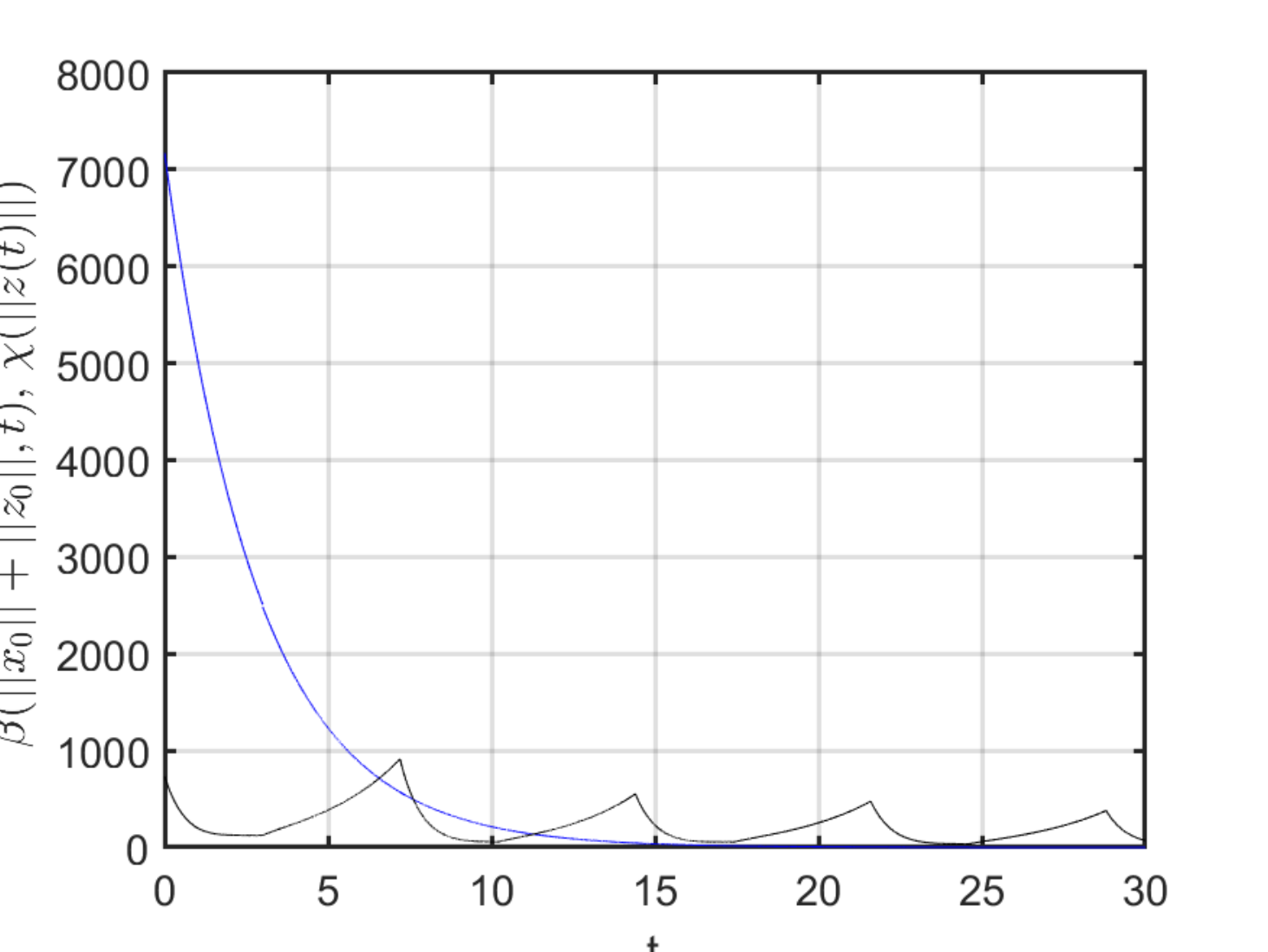}
        \caption{\(\overline{\beta}\bigl(\norm{x(0)}+\norm{z(0)},t\bigr)\) (blue line) and \(\overline{\chi}\bigl(\norm{z(t)}\bigr)\) (black line).}
        \label{fig:funcns}
    \end{minipage}
\end{figure}
%
   
   Next, we fix \(\lambda_s^*=3\), \(\lambda_u^* = 0.75\), and \(\tilde{\delta}=3\), \(\tilde{\Delta} = 4.2\) units of time. It follows that \(\lambda_s^*<\lambda_s\), \(\lambda_u^*>\lambda_u\), \( \lambda_s-\lambda_s^*+\lambda_u-\lambda_u^* = 0.48\),  \(\displaystyle{\tilde\delta<\check\delta,\:\tilde\Delta>\hat\Delta}\) and that conditions \eqref{e:key_condn2}-\eqref{e:key_condn3} and \eqref{e:newcondn4}-\eqref{e:newcondn3} hold with their left-hand sides taking values \(-0.6964\), \(-0.0277\), \(-0.8125\) and \(-0.0857\), respectively.

    We now construct the state-norm estimators \eqref{e:state_norm2}. We choose \(\sigma\in\mathcal{S}'_{\mathcal{R}}\), \(x_0\in[-1,1]\), \({\color{blue}v(t)}\in[-0.5,0.5]\) uniformly at random and \({z}(0) = 2\), and plot the resulting \(\bigl(\norm{x(t)}\bigr)_{t\geq 0}\) and its corresponding state-norm estimation \(\bigl(z(t)\bigr)_{t\geq 0}\) in Fig. \ref{fig:estimate}. 
    The switching signals, \(\sigma, \zeta\) and the functions \(\overline{\beta}\bigl(\norm{x(0)}+\norm{z(0)},t\bigr)\) and \(\overline{\chi}\bigl(\norm{z(t)}\bigr)\) corresponding to Fig. \ref{fig:estimate} appear in Fig. \ref{fig:estimate_sw} and Fig. \ref{fig:funcns}, respectively. The assertion of Theorem \ref{t:mainres3} follows.
%
\section{Concluding remarks}
\label{s:concln}
    We studied IOSS and state-norm estimation of continuous-time switched nonlinear systems under restricted switching. The stabilizing switching signals presented here restrict activation of unstable subsystems consecutively and this offer a scope for further generalization. We identify the design of state-norm estimators for switched systems operating under generalized switching signals in the above sense (\'{a} la involving frequency of switches between certain classes of subsystems \cite{ghi}), as a problem for future work.  
\section{Proofs of our results}
\label{s:all_proofs}
   Let \(]s,t]\subseteq[0,+\infty[\). We let \(\Nsw(s,t)\) denote the total number of switches on \(]s,t]\), \(\NswS(s,t)\) and \(\NswU(s,t)\) denote the total number of switches on \(]s,t]\) where a  stable and an unstable subsystem is activated, respectively, and \(\TswS(s,t)\) and \(\TswU(s,t)\) denote the total duration of activation of stable and unstable subsystems on \(]s,t]\), respectively. We have \(\Nsw(s,t) = \NswS(s,t)+\NswU(s,t)\) and \(t-s = \TswS(s,t)+\TswU(s,t)\).
   The following auxiliary results will be useful.
    \begin{lemma}
    \label{lem:auxres1}
        The following are true:
           a) \(\Big\lfloor\frac{t-s}{\Delta}\Big\rfloor\leq\Nsw(s,t)\leq\Big\lfloor\frac{t-s}{\delta}\Big\rfloor\)
           and b) \(\NswU(s,t)\leq\Big\lfloor\frac{\Nsw(s,t)}{2}\Big\rfloor\).
    \end{lemma}
    \begin{proof}
            a) Recall that \(\check\delta\geq\delta\), \(\hat\Delta\leq\Delta\) and \(\sigma\in\Sw'_{\mathcal{R}}\). Consequently, in view of conditions \ref{prop1}-\ref{prop2}, the minimum and maximum number of switches on \(]s,t]\) are caused when the switching signal, \(\sigma\), dwells on every subsystem for \(\Delta\) and \(\delta\) units of time, respectively.\\ 
            b) In view of condition \ref{prop3}, the maximum number of times an unstable subsystem can be activated by \(\sigma\) on \(]s,t]\) is \(1/2\) the total number of switches caused by \(\sigma\) on \(]s,t]\). 
    \end{proof}
    \begin{lemma}
    \label{lem:auxres2}
        The following are true:
           a) \(\TswS(s,t)\geq\NswS(s,t)\check\delta\) and
           b) \(\TswU(s,t)\leq\NswU(s,t)\hat\Delta+\hat\Delta\).
    \end{lemma}
   \begin{proof}
        Let \(s<\overline{\tau}_1<\overline{\tau}_2<\cdots<\overline{\tau}_{\Nsw(s,t)}\leq t\) be the switching instants of \(\sigma\) on \(]s,t]\). It follows that a) \(\displaystyle{\TswS(s,t) = \mathds{1}_{\bigl\{\sigma(s)\in\P_S\bigr\}}(\overline{\tau}_1-s)
            + \sum_{\substack{{i=1}\\{\sigma(\overline{\tau}_i)\in\P_S}\\{\overline{\tau}_{\Nsw(s,t)+1}:=t}}}^{\Nsw(s,t)} (\overline{\tau}_{i+1}-\overline{\tau}_{i})
            \geq \sum_{\substack{{i=1}\\{\sigma(\overline{\tau}_i)\in\P_S}\\{\overline{\tau}_{\Nsw(s,t)+1}:=t}}}^{\Nsw(s,t)} \check\delta = \NswS(s,t)\check\delta}\), and\\
            b) \(\displaystyle{\TswU(s,t) = \mathds{1}_{\bigl\{\sigma(s)\in\P_U\bigr\}}(\overline{\tau}_1-s)
            +}\displaystyle{\sum_{\substack{{i=1}\\{\sigma(\overline{\tau}_i)\in\P_U}\\{\overline{\tau}_{\Nsw(s,t)+1}:=t}}}^{\Nsw(s,t)} (\overline{\tau}_{i+1}-\overline{\tau}_{i})
            \leq \hat\Delta+\sum_{\substack{{i=1}\\{\sigma(\overline{\tau}_i)\in\P_U}\\{\overline{\tau}_{\Nsw(s,t)+1}:=t}}}^{\Nsw(s,t)} \hat\Delta = \NswU(s,t)\hat\Delta + \hat\Delta}\).
    \end{proof}
    
    We are now ready to present our proof of Theorem \ref{t:mainres}.
      \begin{proof}[Proof of Theorem \ref{t:mainres}]
        Consider a switching signal \(\sigma\in\Sw'_{\mathcal{R}}\). Let Assumptions \ref{assump:key1}-\ref{a:adm-dw} hold. We will show that the switched system \eqref{e:swsys} is IOSS under \(\sigma\).
         
        For a time interval \(]s,t]\subseteq[0,+\infty[\), we define 
        \[
            \Xi(s,t):= -\lambda_s\TswS(s,t)+\lambda_u\TswU(s,t)+(\ln\mu)\Nsw(s,t).
        \]

        Fix \(t>0\). In view of \eqref{e:key-prop2} - \eqref{e:key-prop3}, we have that 
         \begin{align}
         \label{e:pf_step1}
            V_{\sigma(t)}(x(t))&\leq\psi_1(t)V_{\sigma(0)}(x(0))
            + \biggl(\gamma_1\bigl(\norm{v}_{[0,t]}\bigr)+\gamma_2\bigl(\norm{y}_{[0,t]}\bigr)\biggr)\psi_2(t),
         \end{align}
         where
         \begin{align}
         \label{e:psi1_defn}
            \psi_1(t) &= \exp\Bigl(\Xi(0,t)\Bigr),
         \end{align}
         and
         \begin{align}
         \label{e:psi2_defn}
            \psi_2(t) &=\hspace*{-0.1cm}\sum_{\substack{{i=0}\\{\sigma(\tau_i)\in\P_S}\\{\tau_{\Nsw(0,t)+1}:=t}}}^{\Nsw(0,t)}
            \Biggl(\exp\Bigl(\Xi(\tau_{i+1},t)\Bigr)
             \times\frac{1}{\lambda_s}\biggl(1-\exp\bigl(-\lambda_s(\tau_{i+1}-\tau_{i})\bigr)\biggr)\Biggr)\nonumber\\
             &\:\:\hspace*{-0.5cm}-\sum_{\substack{{i=0}\\{\sigma(\tau_i)\in\P_U}\\{\tau_{\Nsw(0,t)+1}:=t}}}^{\Nsw(0,t)}
            \Biggl(\exp\Bigl(\Xi(\tau_{i+1},t)\Bigr)
             \times\frac{1}{\lambda_u}\biggl(1-\exp\bigl(\lambda_u(\tau_{i+1}-\tau_{i})\bigr)\biggr)\Biggr).
         \end{align}
         An application of \eqref{e:key-prop1} to \eqref{e:pf_step1} leads to
         \begin{align*}
            \underline{\alpha}(\norm{x(t)})&\leq\psi_1(t)\overline{\alpha}(\norm{x(0)})
            + \biggl(\gamma_1\bigl(\norm{v}_{[0,t]}\bigr)+\gamma_2\bigl(\norm{y}_{[0,t]}\bigr)\biggr)\psi_2(t).
         \end{align*}
         By Definition \ref{d:ioss-cont}, for IOSS of the switched system \eqref{e:swsys}, we need to show that\\
         i) \(\overline{\alpha}(\star)\psi_1(\cdot)\) is bounded above by a class \(\KL\) function, and\\
         ii) \(\psi_2(\cdot)\) is bounded above by a constant.

         We first verify i). Recall from Assumption \ref{assump:key1} that \(\overline\alpha\in\Kinfty\). Therefore, it suffices to show that \(\psi_1(\cdot)\) is bounded above by a function belonging to class \(\mathcal{L}\). 
         In view of Lemmas \ref{lem:auxres1}-\ref{lem:auxres2}, we have that \(\Xi(0,t)\) is at most equal to
          \begin{align*}
            &\:\:\lambda_u\hat\Delta-\lambda_s\check\delta\Nsw(0,t)+(\lambda_s\check\delta+\lambda_u\hat\Delta)
            \bigg\lfloor\frac{\Nsw(0,t)}{2}\bigg\rfloor+(\ln\mu)\Nsw(0,t)\\
            \leq&\:\:\lambda_u\hat\Delta-\lambda_s\check\delta\bigg\lfloor\frac{t}{\Delta}\bigg\rfloor+(\lambda_s\check\delta+\lambda_u\hat\Delta)
            \bigg\lfloor\frac{t}{2\delta}\bigg\rfloor+(\ln\mu)\bigg\lfloor\frac{t}{\delta}\bigg\rfloor\\
            \leq&\:\:\lambda_u\hat\Delta-\lambda_s\check\delta\biggl(\frac{t}{\Delta}-1\biggr)+
            (\lambda_s\check\delta+\lambda_u\hat\Delta)\biggl(\frac{t}{2\delta}\biggr)+(\ln\mu)\biggl(\frac{t}{\delta}\biggr)\\
            \leq&\:\:(\lambda_u\hat\Delta+\lambda_s\check\delta)+\Biggl(-\lambda_s\frac{\check\delta}{\Delta}+
            \lambda_s\frac{\check\delta}{2\delta}
            +\lambda_u\frac{\hat\Delta}{2\delta}+(\ln\mu)\frac{1}{\delta}\Biggr)t.
         \end{align*}
         By \eqref{e:maincondn}, we have that the above expression is at most equal to \(c_1-c_2t\) for some \(c_1\geq \lambda_u\hat\Delta+\lambda_s\check\delta\) and \(c_2 > 0\). Therefore, from \eqref{e:psi1_defn}, we obtain that
         \(
            \psi_1(t)\leq \exp\bigl(c_1-c_2 t\bigr).
         \)
         The right-hand side of the above inequality decreases as \(t\) increases and tends to \(0\) as \(t\tendsto+\infty\). Consequently, i) holds.

         We next verify ii). We have that \(\psi_2(t)\) is at most equal to
         \begin{align}
         \label{e:pf_step3}
 &= \frac{1}{\lambda_s}\sum_{\substack{{i=0}\\{\sigma(\tau_i)\in\P_S}\\{\tau_{\Nsw(0,t)+1}:=t}}}^{\Nsw(0,t)}
\Biggl(\exp\biggl(-\lambda_s\TswS(\tau_{i+1},t)+\lambda_u\TswU(\tau_{i+1},t)
+(\ln\mu)\Nsw(\tau_{i+1},t)\biggr)\nonumber\\
&\qquad\qquad\qquad\qquad\times\biggl(1-\exp\bigl(-\lambda_s(\tau_{i+1}-\tau_{i})\bigr)\biggr)\Biggr)\nonumber\\
&\:\:-\frac{1}{\lambda_u}\sum_{\substack{{i=0}\\{\sigma(\tau_i)\in\P_U}\\{\tau_{\Nsw(0,t)+1}:=t}}}^{\Nsw(0,t)}
\Biggl(\exp\biggl(-\lambda_s\TswS(\tau_{i+1},t)+\lambda_u\TswU(\tau_{i+1},t)
+(\ln\mu)\Nsw(\tau_{i+1},t)\biggr)\nonumber\\
&\qquad\qquad\qquad\qquad\times\biggl(1-\exp\bigl(\lambda_u(\tau_{i+1}-\tau_{i})\bigr)\biggr)\Biggr)\nonumber\\
&\leq \frac{1}{\lambda_s}\sum_{\substack{{i=0}\\{\sigma(\tau_i)\in\P_S}\\{\tau_{\Nsw(0,t)+1}:=t}}}^{\Nsw(0,t)}
\Biggl(\exp\biggl(-\lambda_s\TswS(\tau_{i+1},t)+\lambda_u\TswU(\tau_{i+1},t)
+(\ln\mu)\Nsw(\tau_{i+1},t)\biggr)\nonumber\\
&\:\:+\frac{1}{\lambda_u}\sum_{\substack{{i=0}\\{\sigma(\tau_i)\in\P_U}\\{\tau_{\Nsw(0,t)+1}:=t}}}^{\Nsw(0,t)}
\Biggl(\exp\biggl(-\lambda_s\TswS(\tau_{i+1},t)+\lambda_u\TswU(\tau_{i+1},t)
+(\ln\mu)\Nsw(\tau_{i+1},t)+\lambda_u(\tau_{i+1}-\tau_{i})\biggr)\nonumber\\
&\leq\frac{1}{\lambda_s}\sum_{\substack{{i=0}\\{\sigma(\tau_i)\in\P_S}\\{\tau_{\Nsw(0,t)+1}:=t}}}^{\Nsw(0,t)}
\exp\biggl(-\lambda_s\TswS(\tau_{i+1},t)+\lambda_u\TswU(\tau_{i+1},t)
+(\ln\mu)\Nsw(\tau_{i+1},t)\biggr)\nonumber\\
&\:\:\:\:+\frac{1}{\lambda_u}\sum_{\substack{{i=0}\\{\sigma(\tau_i)\in\P_U}\\{\tau_{\Nsw(0,t)+1}:=t}}}^{\Nsw(0,t)}
\exp\biggl(-\lambda_s\TswS(\tau_{i},t)+\lambda_u\TswU(\tau_{i},t)
+(\ln\mu)\Nsw(\tau_{i},t)\biggr).
         \end{align}
         Now, fix \(\tau_i\), \(i\in\{0,1,\ldots,\Nsw(0,t)\}\). Applying the set of arguments employed in the verification of i) to the interval \(]\tau_i,t]\), we obtain that
         \begin{align}
         \label{e:pf_step4}
         \hspace*{-0.2cm}\Xi(\tau_i,t)\leq c_1-c_2(t-\tau_i)\:\:\text{and}\:\:\Xi(\tau_{i+1},t)\leq c_1-c_2(t-\tau_{i+1})
         \end{align}
         \(\text{for some}\:c_1\geq\lambda_s\check\delta+\lambda_u\hat\Delta\:\text{and}\:c_2 > 0\).
         An application of \eqref{e:pf_step4} to \eqref{e:pf_step3} leads us to \(\psi_2(t)\) is at most equal to
         \begin{align}
         \label{e:pf_step6}
            \frac{1}{\lambda_s}\sum_{i=0}^{\Nsw(0,t)}\exp\Bigl(c_1-c_2(t-\tau_{i+1})\Bigr)
            +\frac{1}{\lambda_u}\sum_{i=0}^{\Nsw(0,t)}\exp\Bigl(c_1-c_2(t-\tau_{i})\Bigr).
         \end{align}
         We will now show that both the sums in the above expression are bounded above by constants. We have
         \begin{align}
         \label{e:pf_step7}
            &\sum_{i=0}^{\Nsw(0,t)}\exp\biggl(-c_2(t-\tau_{i})\biggr)\nonumber\\
            \leq& \sum_{i=0}^{\big\lfloor\frac{t}{\delta}\big\rfloor}\exp\biggl(-c_2(t-\tau_{i})\biggr)
            \leq\exp\bigl(-c_2t\bigr)+
            \exp\biggl(-c_2\bigg\lfloor\frac{t}{\Delta}\bigg\rfloor\delta\biggr)+
            \cdots+\exp(-c_2\delta)\nonumber\\
            \leq& 1+\frac{1}{\exp\bigl(c_2\delta\bigr)-1}.
         \end{align}
         Applying the set of arguments as above, we also obtain that
         \begin{align}
         \label{e:pf_step8}
            \sum_{i=0}^{\Nsw(0,t)}\exp\biggl(c_1-c_2\bigl(t-\tau_{i+1}\bigr)\biggr)\leq\frac{1}{\exp\bigl(c_2\delta\bigr)-1}.
         \end{align}
         In view of \eqref{e:pf_step7}-\eqref{e:pf_step8}, we arrive at ii). Indeed, \(\psi_2(t)\leq \overline{\psi_2} =\)
         \begin{align*}
             \frac{1}{\lambda_s}\exp(c_1)\frac{1}{\exp(c_2\delta)-1}
            +\frac{1}{\lambda_u}\exp(c_1)\biggl(1+\frac{1}{\exp(c_2\delta)-1}\biggr).
         \end{align*}

         We conclude IOSS of the switched system \eqref{e:swsys} under \(\sigma\). Recall that \(\sigma\in\Sw'_{\mathcal{R}}\) was chosen arbitrarily. Consequently, the assertion of Theorem \ref{t:mainres} follows. In particular, we have that condition \eqref{e:ioss-cont} holds for all \(\sigma\in\Sw'_{\mathcal{R}}\) with \(\alpha(r) = r\), \(\beta(r,s) = \overline{\alpha}(r)\exp\bigl(c_1-c_2s\bigr)\), \(\chi_1(r) = \gamma_1(r)\overline{\psi}_2\) and \(\chi_2(r) = \gamma_2(r)\overline{\psi}_{2}\).
%
    \end{proof}
      \begin{proof}[Proof of Proposition \ref{prop:nextres1}]
             Suppose that \(\Delta\geq 2\delta\). Then \(\displaystyle{\frac{1}{\Delta}\leq\frac{1}{2\delta}}\) implying that \(\check\delta\biggl(\displaystyle{\frac{1}{\Delta}-\frac{1}{2\delta}}\biggr)\leq 0\) (i.e.,
        \(-\lambda_s\check\delta\biggl(\displaystyle{\frac{1}{\Delta}-\frac{1}{2\delta}}\biggr)\geq 0\)). Consequently, the expression on the left-hand side of \eqref{e:maincondn} is non-negative.
         \end{proof}
         \begin{proof}[Proof of Proposition \ref{prop:mainres2}]
       (i) 
        We have \(\mu = 1\), and for any \(\check\delta,\hat\Delta\in[\delta,\Delta]\),
        \begin{align*}
            \:\:-\lambda_s \frac{\check\delta}{\Delta} + \lambda_s \frac{\check\delta}{2\delta} + \lambda_u \frac{\hat\Delta}{2\delta} + \ln\mu \frac{1}{\delta}
            \leq&\:\:-\lambda_s \frac{\delta}{\Delta} + \lambda_s \frac{\Delta}{2\delta} + \lambda_u \frac{\Delta}{2\delta}\\
            =&\:\:\frac{-2\delta^2\lambda_s+\lambda_s\Delta^2+\lambda_u\Delta^2}{2\delta\Delta}\\
            \leq&\:\:\frac{-\varepsilon_1}{2\delta\Delta}\:\:\text{for some}\:\varepsilon_1>0.
        \end{align*}
        Consequently, condition \eqref{e:maincondn} holds.\\
        (ii) 
        We have that for any \(\check\delta,\hat\Delta\in[\delta,\Delta]\),
        \begin{align*}
            -\lambda_s \frac{\check\delta}{\Delta} + \lambda_s \frac{\check\delta}{2\delta} + \lambda_u \frac{\hat\Delta}{2\delta}+\ln\mu\frac{1}{\delta}
            \leq&\:\:-\lambda_s \frac{\delta}{\Delta} + \lambda_s \frac{\Delta}{2\delta} + \lambda_u \frac{\Delta}{2\delta}+\ln\mu\frac{1}{\delta}\\
            =&\:\:-\lambda_s\biggl(\frac{\delta}{\Delta}-\frac{\Delta}{2\delta}\biggr)+\lambda_u\frac{\Delta}{2\delta}
            +\ln\mu\frac{1}{\delta}\\
            \leq&\:\:-\varepsilon_2\:\:\text{for some}\:\varepsilon_2 > 0.
        \end{align*}
        Consequently, condition \eqref{e:maincondn} holds.\\
        (iii) 
        We have
        \begin{align*}
            -\lambda_s \frac{\check\delta}{\Delta} + \lambda_s \frac{\check\delta}{2\delta} + \lambda_u \frac{\hat\Delta}{2\delta} + (\ln\mu) \frac{1}{\delta}
            =&\:\:-\lambda_s \frac{\Delta}{\Delta} + \lambda_s \frac{\Delta}{2\delta} + \lambda_u \frac{\delta}{2\delta} + (\ln\mu) \frac{1}{\delta}\\
            =&\:\:-\lambda_s\biggl(1-\frac{\Delta}{2\delta}\biggr)+\frac{\lambda_u}{2}+(\ln\mu)\frac{1}{\delta}\\
            \leq&-\varepsilon_3\:\:\text{for some}\:\varepsilon_3 > 0.
        \end{align*}
        Consequently, condition \eqref{e:maincondn} holds.\\
         (iv)  
         We have
        \begin{align*}
            -\lambda_s \frac{\check\delta}{\Delta} + \lambda_s \frac{\check\delta}{2\delta} + \lambda_u \frac{\hat\Delta}{2\delta} + (\ln\mu) \frac{1}{\delta}
            =&\:\:-\lambda_s \frac{\delta}{\Delta} + \lambda_s \frac{\delta}{2\delta} + \lambda_u \frac{\Delta}{2\delta} + (\ln\mu) \frac{1}{\delta}\\
            =&\:\:-\lambda_s\biggl(\frac{\delta}{\Delta}-\frac{1}{2}\biggr)+\lambda_u\frac{\Delta}{2\delta}+(\ln\mu)\frac{1}{\delta}\\
            \leq&-\varepsilon_4\:\:\text{for some}\:\varepsilon_4 > 0.
        \end{align*}
        Consequently, condition \eqref{e:maincondn} holds.
           \end{proof}
    The following auxiliary result will be useful in our proof of Theorem \ref{t:mainres3}.
    \begin{lemma}
    \label{lem:auxres3}
        The switched system
    \begin{align}
    \label{e:state_norm1}
        \dot{w}(t) = g_{\upsilon(t)}\bigl(w(t),v(t),y(t)\bigr),\:\:w(0)=w_0\geq 0,\:\:t\in[0,+\infty[
    \end{align}
    consisting of 
        (a) the subsystems \(g_0\bigl(w(t),v(t),y(t)\bigr) = -\lambda_s^*w(t)+\gamma_1\bigl(\norm{v(t)}\bigr)+\gamma_2\bigl(\norm{y(t)}\bigr)\) and \(g_1\bigl(w(t),v(t),y(t)\bigr) = \lambda_u^*w(t)+\gamma_1\bigl(\norm{v(t)}\bigr)+\gamma_2\bigl(\norm{y(t)}\bigr)\),
        where \(\lambda_s^*\), \(\lambda_u^* > 0\) obey conditions \eqref{e:key_condn1a}-\eqref{e:key_condn3}, and (b) the switching signal \(\upsilon:[0,+\infty[\to\{0,1\}\) that obeys \(\upsilon(t) = 0\), if \(\sigma(t)\in\P_S\) and \(\upsilon(t) = 1\), if \(\sigma(t)\in\P_U\), is a state-norm estimator for \eqref{e:swsys} with \(\sigma\in\mathcal{S}'_{\mathrm{R}}\).
    \end{lemma}          
           \begin{proof}
                  Let \(0=:\kappa_0<\kappa_1<\cdots\) be the switching instants of \(\upsilon\). We let \(\Nsw_w(s,t)\), \(\TswS_w(s,t)\) and \(\TswU_w(s,t)\) denote the total number of switches, total duration of activation of the subsystems \(0\) and \(1\) of the switched system \eqref{e:state_norm1} on an interval \(]s,t]\subseteq[0,+\infty[\) of time, respectively. It follows that 
        \(\TswS_w(s,t) = \TswS(s,t)\:\:\text{and}\:\:\TswU_w(s,t) = \TswU(s,t)\:\:\text{for all}\:]s,t]\subseteq[0,+\infty[\).
                We define 
                \begin{align*}
                    \Gamma(s,t) &:= -\lambda_s\TswS(s,t)+\lambda_u\TswU(s,t),\\
                    \Gamma^*(s,t) &:= -\lambda_s^*\TswS(s,t)+\lambda_u^*\TswU(s,t),\\
                    \Gamma_w(s,t) &:= -\lambda_s^*\TswS_w(s,t)+\lambda_u^*\TswU_w(s,t),\\
                    \intertext{and}
                    \overline\Gamma(s,t) &:= -(\lambda_s-\lambda_s^*)\TswS(s,t)+(\lambda_u-\lambda_u^*)\TswU(s,t).
                \end{align*}
                
                Let \(V_0(w) = V_1(w) = \frac{1}{2}w^2\). Clearly, Assumption \ref{assump:key1} holds with the rate of decay (resp., growth) of \(V_0\) (resp., \(V_1\)) along the subsystems dynamics \(g_0\) (resp., \(g_1\)) being \(\lambda_s^*\) (resp., \(\lambda_u^*\)). The set of admissible switches is \(\{(0,1),(1,0)\}\), and Assumption \ref{assump:key2} holds with \(V_0(w) = 1\times V_1(w)\) and \(V_1(w) = 1\times V_0(w)\).
                With \(\mu=1\), following our proof of Theorem \ref{t:mainres}, we will show that 
                A) \(\psi_1^w(t):= \exp\Bigl(\Gamma_w(0,t)\Bigr)\) is bounded above by a function belonging to class \(\mathcal{L}\), and
                B) \(\displaystyle{\psi_2^w(t)\leq \frac{1}{\lambda_s^*}\sum_{\substack{{i=0}\\{\upsilon(\kappa_i)=0}\\\kappa_{\Nsw_w(0,t)+1}:=t}}^{\Nsw_w(0,t)}\exp\Bigl(\Gamma_w(\kappa_{i+1},t)\Bigr) + \frac{1}{\lambda_u^*}\sum_{\substack{{i=0}\\{\upsilon(\kappa_i)=0}\\\kappa_{\Nsw_w(0,t)+1}:=t}}^{\Nsw_w(0,t)}\exp\Bigl(\Gamma_w(\kappa_{i},t)\Bigr)}\) is bounded above by a constant.
                
                Now, 
                \[
                    \Gamma_w(0,t)\leq(\lambda_s^*\check\delta+\lambda_u^*\hat\Delta)+\Bigl(-\lambda_s^*\frac{\check\delta}{\Delta}+\lambda_s^*\frac{\check\delta}{2\delta}+\lambda_u^*\frac{\hat\Delta}{2\delta}\Bigr)t.
                 \]
                 From condition \eqref{e:key_condn2} it follows that 
                 \[
                    \psi_1^w(t)\leq\exp(\overline{c}_1-\overline{c}_2t)\:\:\text{for some}\:\:\overline{c}_1\geq \lambda_s^*\check\delta+\lambda_u^*\hat\Delta\:\:\text{and}\:\:\overline{c}_2 > 0.
                 \] 
                 Clearly, the right-hand side of the above inequality is a class \(\mathcal{L}\) function. Similarly, \[
                    \Gamma_w(\kappa_i,t)\leq \overline{c}_1-\overline{c}_2(t-\kappa_i)\:\:\text{and}\:\: \Gamma_w(\kappa_{i+1},t)\leq \overline{c}_1-\overline{c}_2(t-\kappa_{i+1})
                 \] 
                 for some for some \(\overline{c}_1\geq \lambda_s^*\check\delta+\lambda_u^*\hat\Delta\) and \(\overline{c}_2 > 0\). Under similar arguments employed in our proof of Theorem \ref{t:mainres}, both the sums \(\displaystyle{ \sum_{{i=0}}^{\Nsw_w(0,t)}\exp\Bigl(\overline{c}_1-\overline{c}_2(t-\kappa_{i+1})\Bigr)}\) and \(\displaystyle{ \sum_{{i=0}}^{\Nsw_w(0,t)}\exp\Bigl(\overline{c}_1-\overline{c}_2(t-\kappa_{i})\Bigr)}\) are bounded. Consequently, ISS of the switched system \eqref{e:state_norm1} follows.

                As \(\gamma_1\bigl(\norm{v(t)}\bigr)+\gamma_2\bigl(\norm{y(t)}\bigr)\geq 0\), we have
                \begin{align*}
                    g_0\bigl(w(t),v(t),y(t)\bigr) \geq -\lambda_s^*w(t)\:\:\text{and}\:\:
                    g_1\bigl(w(t),v(t),y(t)\bigr) \geq \lambda_u^*w(t).
                \end{align*}
                Thus, 
                \[
                    w(t)\geq\exp\Bigl(\Gamma_w(0,t)\Bigr)w_0 = \exp\Bigl(\Gamma^*(0,t)\Bigr)w_0.
                \]
                As \(w_0\geq 0\), it follows that 
                    \[
                        w(t)\geq\exp\Bigl(\Gamma^*(0,t)\Bigr)w_0\geq 0\:\text{for all}\:t\in[0,+\infty[.
                    \]
                Moreover, for all \(0\leq\kappa_i\leq t\), we have
                    \[
                        w(t)\geq\exp\Bigl(\Gamma^*(\kappa_i,t)\Bigr)w(\kappa_i)
                    \]
                which can be rewritten as
                \begin{align}
                \label{e:pf2_step1a}
                    w(\kappa_i)\leq\exp\Bigl(-\Gamma^*(\kappa_i,t)\Bigr)w(t).
                \end{align}
                
                We define a function 
                \[
                    W(t) := V_{\sigma(t)}(x(t))-w(t).
                \]
                It follows that
                \begin{align*}
        \dot{W}(t) &= \dot{V}_{\sigma(t)}(x(t))-\dot{w}(t)\\
        &\leq \mathrm{1}_{\{\sigma(t)\in\mathcal{P}_S\}}(-\lambda_s)V_{\sigma(t)}(x(t))+\mathrm{1}_{\{\sigma(t)\in\mathcal{P}_U\}}(\lambda_u)V_{\sigma(t)}(x(t))\\
        &\quad\quad\quad\quad-(\mathrm{1}_{\{\zeta(t)=0\}}(-\lambda_s^*)w(t)+\mathrm{1}_{\{\zeta(t)=1}(\lambda_u^*)w(t))\\
        &= \mathrm{1}_{\{\sigma(t)\in\mathcal{P}_S\}}(-\lambda_s)V_{\sigma(t)}(x(t))+\mathrm{1}_{\{\sigma(t)\in\mathcal{P}_U\}}(\lambda_u)V_{\sigma(t)}(x(t))\\
        &\quad\quad\quad\quad-(\mathrm{1}_{\{\sigma(t)\in\mathcal{P}_S\}}(-\lambda_s^*)w(t)+\mathrm{1}_{\{\sigma(t)\in\mathcal{P}_U\}}(\lambda_u^*)w(t))\\
        &<-\mathrm{1}_{\{\sigma(t)\in\mathcal{P}_S\}}(\lambda_s)V_{\sigma(t)}(x(t))+\mathrm{1}_{\{\sigma(t)\in\mathcal{P}_U\}}(\lambda_u)V_{\sigma(t)}(x(t))\\
        &\quad\quad\quad\quad+\mathrm{1}_{\{\sigma(t)\in\mathcal{P}_S\}}(\lambda_s)w(t)-\mathrm{1}_{\{\sigma(t)\in\mathcal{P}_U\}}(\lambda_u)w(t)\\
        &=\mathrm{1}_{\{\sigma(t)\in\mathcal{P}_S\}}(-\lambda_s)W(t)+\mathrm{1}_{\{\sigma(t)\in\mathcal{P}_U\}}(\lambda_u)W(t).
     \end{align*}
                    
                Now, for any interval \(]\tau_i,\tau_{i+1}]\),\\
                Case I: if \(\upsilon(t) = 0\) on \(]\tau_i,\tau_{i+1}]\), then 
                   \begin{align*}
                    W(\tau_{i+1}) &= V_{\sigma(\tau_{i+1})}\bigl(x(\tau_{i+1})\bigr)-w(\tau_{i+1})\\
                    &\leq \mu V_{\sigma(\tau_i)}\bigl(x(\tau_{i+1})\bigr)-\mu w(\tau_{i+1})+(\mu-1)w(\tau_{i+1})\\
                    &\leq\mu\exp\Bigl(-\lambda_s(\tau_{i+1}-\tau_{i})\Bigr)V_{\sigma(\tau_i)}\bigl(x(\tau_i)\bigr)
                    -\mu\exp\Bigl(-\lambda_s(\tau_{i+1}-\tau_{i})\Bigr)w(\tau_i) + (\mu-1)w(\tau_{i+1})\\
                    &= \mu\exp\Bigl(-\lambda_s(\tau_{i+1}-\tau_{i})\Bigr)W(\tau_i)+(\mu-1)w(\tau_{i+1}).
                \end{align*}
                Case II: if \(\upsilon(t) = 1\) on \(]\tau_i,\tau_{i+1}]\), then 
                 \begin{align*}
                    W(\tau_{i+1}) &= V_{\sigma(\tau_{i+1})}(x(\tau_{i+1})) - w(\tau_{i+1})\\
                    &\leq \mu V_{\sigma(\tau_i)}(x(\tau_{i+1}))-\mu w(\tau_{i+1})+(\mu-1)w(\tau_{i+1})\\
                    &\leq \mu\exp\Bigl(\lambda_u(\tau_{i+1}-\tau_{i})\Bigr)V_{\sigma(\tau_i)}(x(\tau_i))
                    -\mu\exp\Bigl(\lambda_u(\tau_{i+1}-\tau_{i})\Bigr)w(\tau_i)+(\mu-1)w(\tau_{i+1})\\
                    &=\mu\exp\Bigl(\lambda_u(\tau_{i+1}-\tau_i)\Bigr)W(\tau_i)+(\mu-1)w(\tau_{i+1}).
                 \end{align*}
                 Iterating the above from \(i=0\) to \(\Nsw(0,t)\) and applying \eqref{e:pf2_step1a}, we arrive at
                    \[
                        \displaystyle{W(t)\leq \exp\Bigl(\Xi(0,t)\Bigr)W(0)
                     +(\mu-1)w(t)\sum_{i=1}^{\Nsw(0,t)}\biggl(\exp\Bigl(\bigl(\Nsw(0,t)-i\bigr)\ln\mu\Bigr)}\\
                    {\times
                    \exp\Bigl(\overline\Gamma(\tau_i,t)\Bigr)\biggr).}
                    \]
                
                 From the proof of Theorem \ref{t:mainres}, we have that 
                 \[
                    \Xi(0,t)\leq c_1 - c_2 t\:\:\text{for some}\:c_1\geq\lambda_s\check\delta+\lambda_u\hat\Delta\:\text{and}\:c_2>0.
                 \] 
                 Now,
                 \[
                    \displaystyle{\sum_{i=1}^{\Nsw(0,t)}\exp\Bigl((\Nsw(0,t)-i)\ln\mu+\overline\Gamma(\tau_i,t)\Bigr) = \sum_{i=1}^{\Nsw(0,t)}\exp\biggl(\Nsw(\tau_i,t)\ln\mu}
                 {+\overline\Gamma(\tau_i,t)
                   \Bigr)}.
                 \]
                 We have
                 \begin{align*}
                    (\ln\mu)\Nsw(\tau_i,t)+\overline\Gamma(\tau_i,t)
                    &\leq(\lambda_s-\lambda_s^*+\lambda_u-\lambda_u^*)\hat\Delta\\
                    &\quad\quad\quad+\Bigl((\ln\mu)\frac{1}{\delta}-(\lambda_s-\lambda_s^*)
                    +(\lambda_s-\lambda_s^*+\lambda_u-\lambda_u^*)\frac{\hat\Delta}{2\delta}\Bigr)
                    (t-\tau_i).
                 \end{align*}
                 In view of \eqref{e:key_condn3}, the above expression is bounded above by \(\tilde{c}_1-\tilde{c}_2(t-\tau_i)\) for some \(\tilde{c}_1\geq (\lambda_s-\lambda_s^*+\lambda_u-\lambda_u^*)\hat\Delta\) and \(\tilde{c}_2>0\).\\
                 Thus,
                \begin{align*}
                    \displaystyle{\sum_{i=1}^{\Nsw(0,t)}\exp\Bigl(\bigl(\Nsw(0,t)-i\bigr)\ln\mu
                   +\overline\Gamma(\tau_i,t)\Bigr)
                    \leq  \sum_{i=1}^{\Nsw(0,t)}\exp\Bigl(\tilde{c}_1-\tilde{c}_2(t-}{\tau_i)\Bigr)}.
                \end{align*}
                 It is shown in the proof of Theorem \ref{t:mainres} that the above expression is bounded above by 
                 \begin{align*}
                    \exp\bigl(\tilde{c}_1\bigr)\biggl(1+\frac{1}{\exp\tilde{c}_2\delta-1}\biggr) = b\:(\text{say}).
                 \end{align*}
                 Now, 
                \begin{align*}
                    W(t)\leq\exp\Bigl(c_1-c_2t\Bigr)W(0)+\tilde{b}w(t),
                \end{align*}
                 where \(\tilde{b}=(\mu-1)b\). It follows that
                \begin{align*}
                    V_{\sigma(t)}(x(t)) - w(t)\leq\exp\Bigl(c_1-c_2t\Bigr)\Bigl(V_{\sigma(0)}(x(0))-w(0)\Bigr)+\tilde{b}w(t),
                 \end{align*}
                 which implies
                 \begin{align*}
                    V_{\sigma(t)}(x(t))
                    &\leq \exp\Bigl(c_1-c_2t\Bigr)\Biggl(V_{\sigma(0)}(x(0))-w(0)\Biggr)+(1+\tilde{b})w(t)\\
                    &\leq \exp\Bigl(c_1-c_2t\Bigr)\overline{\alpha}_2\Bigl(\norm{x(0)}+\norm{w(0)}\Bigr)+(1+\tilde{b})\norm{w(t)}
                 \end{align*}
                 under the (without loss of generality) assumption that \(\overline{\alpha}_2(r)\geq r\) for all \(r\geq 0\). Using \eqref{e:key-prop1}, we obtain that
                 \begin{align*}
                    \overline{\alpha}_1(\norm{x(t)})\leq\exp\Bigl(c_1-c_2t\Bigr)\overline{\alpha}_2
                    \Bigl(\norm{x(0)}+\norm{w(0)}\Bigr)+(1+\tilde{b})\norm{w(t)},
                 \end{align*}
                 which can be rewritten as
                 \begin{align*}
                    \norm{x(t)} &\leq \overline{\alpha}_1^{-1}\biggl(\exp\Bigl(c_1-c_2t\Bigr)\overline{\alpha}_2\Bigl(\norm{x(0)}+\norm{w(0)}\Bigr)
                    \biggr)
                    +\overline{\alpha}_1^{-1}\biggl((1+\tilde{b})\norm{w(t)}\biggr)\\
                    &=:\overline{\beta}\Bigl(\norm{x_0}+\norm{w_0},t\Bigr)+\overline{\chi}(\norm{w(t)}).
                 \end{align*}
%
           \end{proof}
           System \eqref{e:state_norm1} requires knowledge of exact switching instants and destinations of \(\sigma\in\mathcal{S}'_{\mathcal{R}}\), and is not a robust estimator.
           \begin{proof}[Proof of Theorem \ref{t:mainres3}]
                In view of Definition \ref{d:ioss-cont}, we need to establish the following two properties of the switched system \eqref{e:state_norm2}:\\
                i) It is ISS with respect to \((v,y)\), and\\
                ii) condition \eqref{e:norm_condn} holds.
               
               We first verify i). Let \(V_0({z}) = V_1({z}) = \frac{1}{2}{z}^2\). Clearly, Assumption \ref{assump:key1} holds with the rate of decay (resp., growth) of \(V_0\) (resp., \(V_1\)) along the subsystems dynamics \(g_0\) (resp., \(g_1\)) being \(\lambda_s^*\) (resp., \(\lambda_u^*\)). The set of admissible switches is \(\{(0,1),(1,0)\}\), and Assumption \ref{assump:key2} holds with \(V_0({z}) = 1\times V_1({z})\) and \(V_1({z}) = 1\times V_0({z})\).\\
               Let \(0=:\overline\kappa_0<\overline\kappa_1<\cdots\) be the switching instants of \(\zeta\). Let \({\Nsw}_{{z}}(s,t)\), \(\TswS_{{z}}(s,t)\) and \(\TswU_{{z}}(s,t)\) denote the number of switches of \(\zeta\), the total duration of activation of stable subsystems by \(\zeta\) and the total duration of activation of unstable subsystems by \({\zeta}\) on an interval \(]s,t]\subseteq[0,+\infty[\), respectively.\\
               We will follow our proof of Theorem \ref{t:mainres}. Notice that \(\ln\mu = 0\). Define 
               \[
                \Theta(s,t):= -\lambda_s^*\TswS_{{z}}(s,t)
                        +\lambda_u^*\TswU_{{z}}(s,t).
                \]
                We need to show that\\
                   A) \({\psi}_{1}^{{z}}(t):= \exp\Bigl(\Theta(0,t)\Bigr)\) is bounded above by a function belonging to class \(\mathcal{L}\), and\\
                   B) \(\displaystyle{\psi_2^{{z}}(t) \leq\frac{1}{\lambda_s^*}\sum_{\substack{{i=0}\\{\zeta(\overline\kappa_i)=0}\\
                        {\overline\kappa_{{\Nsw}_z(0,t)+1}:=t}}}^{{\Nsw}_{{z}}(0,t)}
\exp\Bigl(\Theta(\overline\kappa_{i+1},t)\Bigr)}
{+\frac{1}{\lambda_u^*}\sum_{\substack{{i=0}\\{\zeta(\overline\kappa_i)=1}\\
{\overline\kappa_{{\Nsw}_z(0,t)+1}:=t}}}^{{\Nsw}_{{z}}(0,t)}
\exp\Bigl(\Theta(\overline\kappa_i,t)\Bigr)}\) is bounded above by a constant.

               Now, 
               \begin{align*}
                    \Theta(0,t) = -\lambda_s^*t+(\lambda_s^*+\lambda_u^*)\TswU_{{z}}(0,t)
                    &\leq\:\: -\lambda_s^*t + (\lambda_s^*+\lambda_u^*)2\Bigl\lfloor\frac{t}{\tilde\delta+\tilde\Delta}\Bigr\rfloor\frac{1}{2}\tilde\Delta\\
                     &\leq \Biggl(-\lambda_s^*+(\lambda_s^*+\lambda_u^*)\frac{\tilde\Delta}{\tilde\delta+\tilde\Delta}\Biggr)t.
               \end{align*} 
               By \eqref{e:newcondn4}, we have that the above expression is at most equal to \(-\overline{c}t\) for some \(\overline{c}>0\). Thus, \({\psi}_1^{{z}}(t)\leq\exp\bigl(-\overline{c}t\bigr)\). Clearly, the right-hand side of the above inequality is a class \(\mathcal{L}\) function.
               Similarly, 
               \[
                \Theta(\overline\kappa_i,t)\leq\overline{c}(t-\overline\kappa_i)\:\:\text{and}\:\: \Theta(\overline\kappa_{i+1},t)\leq\overline{c}(t-\overline\kappa_{i+1})\:\:\text{for some}\:\: \overline{c}>0.
               \]
Under a similar set of arguments as employed in our proof of Theorem \ref{t:mainres}, both the sums 
               \(\displaystyle{\sum_{i=0}^{{\Nsw}_{{z}}(0,t)}\exp\biggl(-\overline{c}(t-\overline{\kappa}_i)\biggr)}\) and
               \(\displaystyle{\sum_{i=0}^{{\Nsw}_{{z}}(0,t)}\exp\biggl(-\overline{c}(t-\overline{\kappa}_{i+1})\biggr)}\) are bounded. 
               Consequently, ISS of the switched system \eqref{e:state_norm2} follows.
                 
                We next verify ii). Recall the state-norm estimator \eqref{e:state_norm1} in Lemma \ref{lem:auxres3}. We will employ the following fact in our proof: If the state, \({z}\), of the system \eqref{e:state_norm2} and the state, \(w\), of the system \eqref{e:state_norm1} are related as follows:
                \begin{align}
                \label{e:pf2_step1}
                    \norm{w(t)}\leq c\norm{{z}(t)}\:\text{for all}\:t\geq 0\:\text{and some constant}\:c\geq 1,
                \end{align}
                then the system \eqref{e:state_norm2} satisfies condition (b) in Definition \ref{d:norm_est}. In particular, 
                \begin{align*}
                    \norm{x(t)}\leq\overline\chi\Bigl(c\norm{{z}(t)}\Bigr)+\overline\beta\Bigl(\norm{x_0}+\norm{{z}_0},t\Bigr)\:\text{for all}\:t\geq 0.
               \end{align*}
                
                We will prove \eqref{e:pf2_step1} by employing mathematical induction. We can rewrite \eqref{e:pf2_step1} as 
                    \(w(t)\leq c{z}(t)\:\text{for all}\:t\geq 0\:\text{and some constant}\:c\geq 1\)
                because \(w(t)\) and \({z}(t)\) are greater than or equal to \(0\) for all \(t\geq 0\).
                
                Let us choose \(w_0\) and \({z}_0\) such that \(w_0\leq{z}_0\). Assume that 
                \begin{align}
                \label{e:pf2_step2}
                    w\bigl(k(\tilde\delta+\tilde\Delta)\bigr)\leq{z}\bigl(k(\tilde\delta+\tilde\Delta)\bigr)\:
                    \text{for some}\:k\in\N_0.
                \end{align}
                Towards proving that \eqref{e:pf2_step1} holds, it suffices to show that
                \begin{align}
                \label{e:pf2_step3}
                    w\bigl((k+1)(\tilde\delta+\tilde\Delta)\bigr)\leq{z}\bigl((k+1)(\tilde\delta+\tilde\Delta)\bigr),
                \end{align}
                and
                \begin{align}
                \label{e:pf2_step4}
                    w(t)\leq c{z}(t)\:\text{with}\:k(\tilde\delta+\tilde\Delta)\leq t\leq(k+1)(\tilde\delta+\tilde\Delta)
                \end{align}
                for some constant \(c\geq 1\).
                
                We let \(t_1:=k\bigl(\tilde\delta+\tilde\Delta\bigr)\), \(t_2:= k\bigl(\tilde\delta+\tilde\Delta\bigr)+\tilde\delta\) and 
                \(t_3:= (k+1)\bigl(\tilde\delta+\tilde\Delta\bigr)\). Let \(\overline{\gamma}(t) := \gamma_1(\norm{v(t)})+\gamma_2(\norm{y(t)})\).
                
               Consider the interval \([t_1,t_2[\). Here, \({z}\) follows the ISS dynamics, \(g_0\) and \(w\) may follow the ISS dynamics, \(g_0\) and the unstable dynamics, \(g_1\) over disjoint intervals. 
                
                Consider an interval \([t',t''[\), where \(t_1\leq t'\leq t''\leq t_2\), and both \({z}\) and \(w\) follow the ISS dynamics, \(g_0\), and \(w(t')\leq\varepsilon{z}(t')\) for some \(\varepsilon\geq 1\). We have 
               \begin{align*}
                    \frac{d}{dt}\Bigl(w(t)-{z}(t)\Bigr) = g_0\Bigl(w(t),v(t),y(t)\Bigr)-g_0\Bigl({z}(t),v(t),y(t)\Bigr)
                    =-\lambda_s^*\Bigl(w(t)-{z}(t)\Bigr).
                \end{align*}
                It follows that
                \begin{align*}
                    w\bigl(t''\bigr)-{z}\bigl(t''\bigr) &= \exp\Bigl(-\lambda_s^*(t''-t')\Bigr)w\bigl(t'\bigr)-\exp\Bigl(-\lambda_s^*(t''-t')\Bigr){z}\bigl(t'\bigr)\\
                    &\leq\exp\Bigl(-\lambda_s^*\bigl(t''-t'\bigr)\Bigr)(\varepsilon-1)\exp\Bigl(\lambda_s^*\bigl(t''-t'\bigr)\Bigr){z}\bigl(t''\bigr)\\
                    &=(\varepsilon-1){z}\bigl(t''\bigr).
                \end{align*}
                The above inequality can be rewritten as
                \begin{align}
                \label{e:pf2_step5}
                    w\bigl(t''\bigr)\leq\varepsilon{z}\bigl(t''\bigr).
                \end{align}
                Consider an interval \([t'',t'''[\), where \(t_1\leq t''\leq t'''\leq t_2\), and \({z}\) follows the ISS dynamics, \(g_0\), \(w\) follows the unstable dynamics, \(g_1\), and
                    \(w\bigl(t''\bigr)\leq\overline\varepsilon{z}\bigl(t''\bigr)\:\text{for some}\:\overline\varepsilon\geq 1.\)
                We have
                \begin{align*}
        z(t''') &= \exp\Bigl(-\lambda_s^*(t'''-t'')\Bigr)z(t'')+\int_{t''}^{t'''}\exp\Bigl(-\lambda_s^*(t'''-s)\Bigr)\overline\gamma(s)ds\\
        &= \exp\Bigl(-\lambda_s^*(t'''-t'')\Bigr)z(t'')+\int_{t''}^{t'''}\exp\Bigl(-\lambda_s^*(t'''-t'')+\lambda_s^*(s-t'')\Bigr)\overline\gamma(s)ds\\
        &= \exp\Bigl(-\lambda_s^*(t'''-t'')\Bigr)\Biggl(z(t'')+\int_{t''}^{t'''}\exp\Bigl(\lambda_s^*(s-t'')\Bigr)\overline\gamma(s)ds\Biggr)\\
        &\geq \exp\Bigl(-\lambda_s^*(t'''-t'')\Bigr)\Biggl(z(t'')+\int_{t''}^{t'''}\overline\gamma(s)ds\Biggr),
     \end{align*}
     and
     \begin{align}
     \label{e:pf2_step6}
        w(t''') &= \exp\Bigl(\lambda_u^*(t'''-t'')\Bigr)w(t'')+\int_{t''}^{t'''}\exp\Bigl(\lambda_u^*(t'''-s)\Bigr)\overline\gamma(s)ds\nonumber\\
        &= \exp\Bigl(\lambda_u^*(t'''-t'')\Bigr)w(t'')+\int_{t''}^{t'''}\exp\Bigl(\lambda_u^*(t'''-t'')-\lambda_u^*(s-t'')\Bigr)\overline\gamma(s)ds\nonumber\\
        &= \exp\Bigl(\lambda_u^*(t'''-t'')\Bigr)\Biggl(w(t'')+\int_{t''}^{t'''}\exp\Bigl(-\lambda_u^*(s-t'')\Bigr)\overline\gamma(s)ds\Biggr)\nonumber\\
        &\leq \exp\Bigl(\lambda_u^*(t'''-t'')\Bigr)\Biggl(w(t'')+\int_{t''}^{t'''}\overline\gamma(s)ds\Biggr)\nonumber\\
        &\leq \exp\Bigl(\lambda_u^*(t'''-t'')\Bigr)\Biggl(\overline\varepsilon z(t'')+\int_{t''}^{t'''}\overline\gamma(s)ds\Biggr)\nonumber\\
        &= \exp\Bigl(\lambda_u^*(t'''-t'')\Bigr)\Biggl(\overline\varepsilon z(t'')+\overline\varepsilon\frac{1}{\overline\varepsilon}\int_{t''}^{t'''}\overline\gamma(s)ds\Biggr)\nonumber\\
        &= \exp\Bigl(\lambda_u^*(t'''-t'')+\lambda_s^*(t'''-t'')-\lambda_s^*(t'''-t'')\Bigr)\overline\varepsilon\Biggl( z(t'')+\frac{1}{\overline\varepsilon}\int_{t''}^{t'''}\overline\gamma(s)ds\Biggr)\nonumber\\
        &\leq \exp\Bigl(\lambda_s^*+\lambda_u^*\Bigr)(t'''-t'')\overline\varepsilon\exp\Bigl(-\lambda_s^*(t'''-t'')\Bigr)\Biggl(z(t'')+\int_{t''}^{t'''}\overline\gamma(s)ds\Biggr)\nonumber\\
        &\leq \exp\Bigl((\lambda_s^*+\lambda_u^*)(t'''-t'')\Bigr)\overline\varepsilon z(t'''),
     \end{align}
                From \eqref{e:pf2_step2}, we have that \(w(t_1)\leq{z}(t_1)\). An iteration of \eqref{e:pf2_step5} and \eqref{e:pf2_step6} leads us to
                \begin{align}
                \label{e:pf2_step7}
                    w(t_2)\leq\exp\Bigl((\lambda_u^*+\lambda_s^*)\TswU(t_1,t_2)\Bigr){z}(t_2).
                \end{align}
                
                Now, consider the interval \([t_2,t_3[\). Here, \({z}\) follows the unstable dynamics, \(g_1\) and \(w\) may follow the ISS dynamics, \(g_0\) and the unstable dynamics, \(g_1\) over disjoint intervals. We have
                \begin{align}
                \label{e:pf2_step8}
                    \TswU_w(t_2,t_3) &= \TswU(t_2,t_3) \leq \NswU(t_2,t_3)\hat\Delta
                    \leq \frac{\tilde\Delta}{2\delta}\hat\Delta,\\
                    \intertext{and}
                    \label{e:pf2_step10}\TswU_w(t_1,t_2) &= \TswU(t_1,t_2) \leq \NswU(t_1,t_2)\hat\Delta
                    \leq\frac{\tilde\delta}{2\delta}\hat\Delta.
                \end{align}
                
                Suppose that there are \(\ell\) switching instants of \(\sigma\) during the interval \(]t_2,t_3[\). Assume without loss of generality that these are the switching instants \(\tau_{r+1}\), \(\tau_{r+2}\), \(\ldots,\tau_{r+\ell}\), where \(r\in\N_0\). Further, without loss of generality, let \(\tau_r:= t_2\) and \(\tau_{r+\ell+1}:=t_3\). Let 
                \begin{align*}
                    \mathcal{T}_1 := \bigl\{k\in\{0,1,\ldots,\ell\}\:|\:\zeta(t) = 0\:\text{for}\:t\in[\tau_{r+k},\tau_{r+k+1}[\bigr\},\\
                    \intertext{and}                      
                    \mathcal{T}_2 := \bigl\{k\in\{0,1,\ldots,\ell\}\:|\:\zeta(t) = 1\:\text{for}\:t\in[\tau_{r+k},\tau_{r+k+1}[\bigr\}.
                \end{align*}
                We have 
                    \begin{align}
                    \label{e:pf2_step11}
        z(t_3) &= \exp\Bigl(\lambda_u^*(t_3-t_2)\Bigr)\Biggl(z(t_2)+\int_{t_2}^{t_3}\exp\Bigl(-\lambda_u^*(s-t_2)\Bigr)\overline\gamma(s)ds\Biggr)\nonumber\\
        &= \exp\Bigl(\lambda_u^*(t_3-t_2)\Bigr)z(t_2)+\sum_{k\in\mathcal{T}_1}\exp\Bigl(\lambda_u^*(t_3-\tau_{r+k+1})\Bigr)\int_{\tau_{r+k}}^{\tau_{r+k+1}}\exp\Bigl(\lambda_u^*(\tau_{r+k+1}-s)\Bigr)\overline\gamma(s)ds\nonumber\\
        &\quad\quad\quad\quad+\sum_{k\in\mathcal{T}_2}\exp\Bigl(\lambda_u^*(t_3-\tau_{r+k+1})\Bigr)\int_{\tau_{r+k}}^{\tau_{r+k+1}}\exp\Bigl(\lambda_u^*(\tau_{r+k+1}-s)\Bigr)\overline\gamma(s)ds\nonumber\\
        &=  \exp\Bigl(\lambda_u^*(t_3-t_2)\Bigr)z(t_2)+\sum_{k\in\mathcal{T}_1}\exp\Bigl(\lambda_u^*(t_3-\tau_{r+k+1})\Bigr)\int_{\tau_{r+k}}^{\tau_{r+k+1}}\exp\Bigl(\lambda_u^*(\tau_{r+k+1}-s)\Bigr)\overline\gamma(s)ds\nonumber\\
        &\quad\quad\quad\quad+\sum_{k\in\mathcal{T}_2}\int_{\tau_{r+k}}^{\tau_{r+k+1}}\exp\Bigl(\lambda_u^*(t_3-\tau_{r+k})-\lambda_u^*(s-\tau_{r+k})\Bigr)\overline\gamma(s)ds\nonumber\\
        &\geq  \exp\Bigl(\lambda_u^*(t_3-t_2)\Bigr)z(t_2)+\sum_{k\in\mathcal{T}_1}\exp\Bigl(\lambda_u^*(t_3-\tau_{r+k+1})\Bigr)\int_{\tau_{r+k}}^{\tau_{r+k+1}}\overline\gamma(s)ds\nonumber\\
        &\quad\quad\quad\quad+\sum_{k\in\mathcal{T}_2}\exp\Bigl(\lambda_u^*(t_3-\tau_{r+k})\Bigr)\int_{\tau_{r+k}}^{\tau_{r+k+1}}\exp\Bigl(-\lambda_u^*(s-\tau_{r+k})\Bigr)\overline\gamma(s)ds\nonumber\\
        &= \exp\Bigl(\lambda_u^*(t_3-t_2)\Bigr)z(t_2)+a_1+a_2,
     \end{align}
     and
     \begin{align}
     \label{e:pf2_step11}
        w(t_3) &= \exp\Bigl(-\lambda_s^*\mathrm{T}^{\mathrm{S}}(t_2,t_3)+\lambda_u^*\mathrm{T}^{\mathrm{U}}(t_2,t_3)\Bigr)w(t_2)
        + \sum_{k=0}^{\ell}\exp\Bigl(-\lambda_s^*\mathrm{T}^{\mathrm{S}}(\tau_{r+k+1},t_3)+\lambda_u^*\mathrm{T}^{\mathrm{U}}(\tau_{r+k+1},t_3)\Bigr)\nonumber\\
        &\quad\quad\times\int_{\tau_{r+k}}^{\tau_{r+k+1}}\exp\Bigl(\mathrm{1}_{\{\sigma(\tau_{r+k})\in\mathcal{P}_S\}}(-\lambda_s^*)(\tau_{r+k+1}-s)
        +\mathrm{1}_{\{\sigma(\tau_{r+k})\in\mathcal{P}_U\}}(\lambda_u^*)(\tau_{r+k+1}-s)\Bigr)\overline\gamma(s)ds\nonumber\\
        &= \exp\Bigl(-\lambda_s^*\mathrm{T}^{\mathrm{S}}(t_2,t_3)+\lambda_u^*\mathrm{T}^{\mathrm{U}}(t_2,t_3)\Bigr)w(t_2)\nonumber\\
        &\:\:+\sum_{k\in\mathcal{T}_1}\exp\Bigl(-\lambda_s^*\mathrm{T}^{\mathrm{S}}(\tau_{r+k+1},t_3)+\lambda_u^*\mathrm{T}^{\mathrm{U}}(\tau_{r+k+1},t_3)\Bigr)
        \times\int_{\tau_{r+k}}^{\tau_{r+k+1}}\exp\Bigl(-\lambda_s^*(\tau_{r+k+1}-s)\Bigr)\overline\gamma(s)ds\nonumber\\
        &\:\:+ \sum_{k\in\mathcal{T}_2}\exp\Bigl(-\lambda_s^*\mathrm{T}^{\mathrm{S}}(\tau_{r+k+1},t_3)+\lambda_u^*\mathrm{T}^{\mathrm{U}}(\tau_{r+k+1},t_3)\Bigr)
        \times\int_{\tau_{r+k}}^{\tau_{r+k+1}}\exp\Bigl(\lambda_u^*(\tau_{r+k+1}-s)\Bigr)\overline\gamma(s)ds\nonumber\\
        &= \exp\Bigl(-\lambda_s^*\mathrm{T}^{\mathrm{S}}(t_2,t_3)+\lambda_u^*\mathrm{T}^{\mathrm{U}}(t_2,t_3)\Bigr)w(t_2)\nonumber\\
        &\:\:+\sum_{k\in\mathcal{T}_1}\exp\Bigl(-\lambda_s^*\mathrm{T}^{\mathrm{S}}(\tau_{r+k+1},t_3)+\lambda_u^*\mathrm{T}^{\mathrm{U}}(\tau_{r+k+1},t_3)-\lambda_s^*(\tau_{r+k+1}-\tau_{r+k})\Bigr)\nonumber\\
        &\hspace*{2cm}\times\int_{\tau_{r+k}}^{\tau_{r+k+1}}\exp\Bigl(\lambda_s^*(s-\tau_{r+k})\Bigr)\overline\gamma(s)ds\nonumber\\
        &\:\:+\sum_{k\in\mathcal{T}_2}\exp\Bigl(-\lambda_s^*\mathrm{T}^{\mathrm{S}}(\tau_{r+k+1},t_3)+\lambda_u^*\mathrm{T}^{\mathrm{U}}(\tau_{r+k+1},t_3)+\lambda_u^*(\tau_{r+k+1}-\tau_{r+k})\Bigr)\nonumber\\
        &\hspace*{2cm}\times\int_{\tau_{r+k}}^{\tau_{r+k+1}}\exp\Bigl(-\lambda_u^*(s-\tau_{r+k})\Bigr)\overline\gamma(s)ds\nonumber\\
        &\leq \exp\Bigl(-\lambda_s^*\mathrm{T}^{\mathrm{S}}(t_2,t_3)+\lambda_u^*\mathrm{T}^{\mathrm{U}}(t_2,t_3)\Bigr)w(t_2) + a_2 + a_1.
     \end{align}
                 Now,
                 \begin{align*}
                    {z}(t_3)-w(t_3) &\geq \exp\Bigl(\lambda_u^*(t_3-t_2)\Bigr){z}(t_2)
                    -\exp\Bigl(\Gamma^*(t_2,t_3)\Bigr)w(t_2)\\
                    &\geq \Biggl(\exp\Bigl(\lambda_u^*(t_3-t_2)\Bigr)
                    -\exp\Bigl(\Gamma^*(t_2,t_3)+(\lambda_s^*+\lambda_u^*)\TswU(t_1,t_2)\Bigr)\Biggr){z}(t_2)\\
                    &\geq \Biggl(\exp\Bigl(\lambda_u^*(t_3-t_2)\Bigr)
                    -\exp\Bigl(-\lambda_s^*(t_3-t_2)
                    +(\lambda_s^*+\lambda_u^*)\TswU(t_2,t_3)
                   +(\lambda_s^*+\lambda_u^*)\TswU(t_1,t_2)\Bigr)\Biggr){z}(t_2).
                 \end{align*}
                 Applying \eqref{e:pf2_step8}-\eqref{e:pf2_step10}, we obtain that the right-hand side of the above inequality is at least equal to
                 \begin{align*}
                    \Biggl(\exp\Bigl(\lambda_u^*\tilde\Delta\Bigr)-\exp\biggl(-\lambda_s^*\tilde\Delta
                    +(\lambda_s^*+\lambda_u^*)\frac{\tilde\Delta}{2\delta}\hat\Delta
                    +(\lambda_s^*+\lambda_u^*)\frac{\tilde\delta}{2\delta}\hat\Delta\biggr)\Biggr){z}(t_2).
                 \end{align*}
                 Since \({z}(t_2)\geq 0\), non-negativity of the above expression is ensured by 
                 condition \eqref{e:newcondn3}. 
                 
                 Therefore, we have shown that \eqref{e:pf2_step3} and \eqref{e:pf2_step4} are true. The constant \(c\) in \eqref{e:pf2_step4} can be computed by employing \eqref{e:pf2_step7} and the upper bound on \(\TswU(t_1,t_2)\) as:
                    \(c = \exp\Biggl((\lambda_u^*+\lambda_s^*)\biggl(\frac{\check\delta\hat\Delta}{2\delta}+\hat\Delta\biggr)\Biggr).
                    \)
                 We note that computing \(c\) from \eqref{e:pf2_step7} is no loss of generality as the worst-case ratio \(\displaystyle{\frac{w(t)}{{z}(t)}}\) can be calculated at \(t=t_2\). Indeed, during the interval \([t_2,t_3[\), the ratio \(\displaystyle{\frac{w(t)}{{z}(t)}}\) decreases. Here, \({z}\) follows the unstable dynamics, \(g_1\) and \(w\) follows the ISS dynamics, \(g_0\) and the unstable dynamics, \(g_1\) over disjoint intervals. We have 
                 \begin{align*}
                    \frac{d}{dt}\Bigl(w(t)-{z}(t)\Bigr) \leq \lambda_u^*\Bigl(w(t)-{z}(t)\Bigr).
                 \end{align*}
                 If \(w(t_2) = \tilde\varepsilon{z}(t_2)\) for some \(\tilde{\varepsilon}\), then it follows that for any \(t\in[t_2,t_3[\), 
                 \begin{align*}
                    w(t)-{z}(t) = \exp\Bigl(\lambda_u^*(t-t_2)\Bigr)\Bigl(w(t_2)-{z}(t_2)\Bigr)
                    \leq \exp\Bigl(-\lambda_s^*(t-t_2)\Bigr)(\varepsilon-1){z}(t_2)
                    \leq(\varepsilon-1){z}(t_2),
                 \end{align*}
                 which can be rewritten as \(w(t)\leq\tilde\varepsilon{z}(t)\), i.e., the ratio \(\displaystyle{\frac{w(t)}{{z}(t)}}\) is smaller than or equal to the ratio \(\displaystyle{\frac{w(t_2)}{{z}(t_2)}}\).
                 
                 We conclude that the system \eqref{e:state_norm2} is a state-norm estimator for the switched system \eqref{e:swsys} operating under \(\sigma\in\Sw'_{\mathcal{R}}\).
           \end{proof}

\end{document}